\documentclass[oneside,english,reqno]{amsart}
\usepackage[T1]{fontenc}
\usepackage[latin9]{inputenc}
\usepackage{babel}
\usepackage{array}
\usepackage{refstyle}
\usepackage{float}
\usepackage{mathrsfs}
\usepackage{mathtools}
\usepackage{enumitem}
\usepackage{amstext}
\usepackage{amsthm}
\usepackage{amssymb}
\usepackage{graphicx}
\usepackage[all]{xy}
\PassOptionsToPackage{normalem}{ulem}
\usepackage{ulem}
\usepackage[unicode=true,pdfusetitle,
 bookmarks=true,bookmarksnumbered=false,bookmarksopen=false,
 breaklinks=false,pdfborder={0 0 1},backref=false,colorlinks=false]
 {hyperref}
\hypersetup{
 colorlinks=true,citecolor=blue,linkcolor=blue,linktocpage=true}

\makeatletter

\AtBeginDocument{\providecommand\subsecref[1]{\ref{subsec:#1}}}
\AtBeginDocument{\providecommand\remref[1]{\ref{rem:#1}}}
\AtBeginDocument{\providecommand\secref[1]{\ref{sec:#1}}}
\AtBeginDocument{\providecommand\exaref[1]{\ref{exa:#1}}}
\AtBeginDocument{\providecommand\thmref[1]{\ref{thm:#1}}}
\AtBeginDocument{\providecommand\defref[1]{\ref{def:#1}}}
\AtBeginDocument{\providecommand\lemref[1]{\ref{lem:#1}}}
\AtBeginDocument{\providecommand\figref[1]{\ref{fig:#1}}}
\AtBeginDocument{\providecommand\propref[1]{\ref{prop:#1}}}
\AtBeginDocument{\providecommand\corref[1]{\ref{cor:#1}}}
\providecommand{\tabularnewline}{\\}
\RS@ifundefined{subsecref}
  {\newref{subsec}{name = \RSsectxt}}
  {}
\RS@ifundefined{thmref}
  {\def\RSthmtxt{theorem~}\newref{thm}{name = \RSthmtxt}}
  {}
\RS@ifundefined{lemref}
  {\def\RSlemtxt{lemma~}\newref{lem}{name = \RSlemtxt}}
  {}

\numberwithin{equation}{section}
\numberwithin{figure}{section}
\numberwithin{table}{section}
\theoremstyle{plain}
\newtheorem{thm}{\protect\theoremname}[section]
  \theoremstyle{definition}
  \newtheorem{defn}[thm]{\protect\definitionname}
  \theoremstyle{remark}
  \newtheorem{rem}[thm]{\protect\remarkname}
  \theoremstyle{definition}
  \newtheorem{example}[thm]{\protect\examplename}
  \theoremstyle{plain}
  \newtheorem{lem}[thm]{\protect\lemmaname}
  \theoremstyle{plain}
  \newtheorem{prop}[thm]{\protect\propositionname}
  \theoremstyle{plain}
  \newtheorem{cor}[thm]{\protect\corollaryname}
  \theoremstyle{plain}
  \newtheorem{question}[thm]{\protect\questionname}
  \theoremstyle{remark}
  \newtheorem*{acknowledgement*}{\protect\acknowledgementname}

\setlist[enumerate]{itemsep=5pt,topsep=3pt}
\setlist[enumerate,1]{label=(\arabic*),ref=\arabic*}
\setlist[enumerate,2]{label=(\alph*),ref=\theenumi \alph*}

\usepackage{needspace}
\usepackage{bbm}
\allowdisplaybreaks

\newref{lem}{refcmd={Lemma \ref{#1}}}
\newref{thm}{refcmd={Theorem \ref{#1}}}
\newref{cor}{refcmd={Corollary \ref{#1}}}
\newref{sec}{refcmd={Section \ref{#1}}}
\newref{sub}{refcmd={Section \ref{#1}}}
\newref{subsec}{refcmd={Section \ref{#1}}}
\newref{chap}{refcmd={Chapter \ref{#1}}}
\newref{prop}{refcmd={Proposition \ref{#1}}}
\newref{exa}{refcmd={Example \ref{#1}}}
\newref{tab}{refcmd={Table \ref{#1}}}
\newref{rem}{refcmd={Remark \ref{#1}}}
\newref{def}{refcmd={Definition \ref{#1}}}
\newref{fig}{refcmd={Figure \ref{#1}}}
\newref{claim}{refcmd={Claim \ref{#1}}}

\@ifundefined{showcaptionsetup}{}{%
 \PassOptionsToPackage{caption=false}{subfig}}
\usepackage{subfig}
\AtBeginDocument{
  
}

\makeatother

  \providecommand{\acknowledgementname}{Acknowledgement}
  \providecommand{\corollaryname}{Corollary}
  \providecommand{\definitionname}{Definition}
  \providecommand{\examplename}{Example}
  \providecommand{\lemmaname}{Lemma}
  \providecommand{\propositionname}{Proposition}
  \providecommand{\questionname}{Question}
  \providecommand{\remarkname}{Remark}
\providecommand{\theoremname}{Theorem}

\begin{document}

\title{Dynamical properties of endomorphisms, multiresolutions, similarity-and
orthogonality relations }

\author{Palle Jorgensen and Feng Tian}

\address{(Palle E.T. Jorgensen) Department of Mathematics, The University
of Iowa, Iowa City, IA 52242-1419, U.S.A. }

\email{palle-jorgensen@uiowa.edu}

\urladdr{http://www.math.uiowa.edu/\textasciitilde{}jorgen/}

\address{(Feng Tian) Department of Mathematics, Hampton University, Hampton,
VA 23668, U.S.A.}

\email{feng.tian@hamptonu.edu}



\maketitle
\pagestyle{myheadings}
\markright{}
\begin{abstract}
We study positive transfer operators $R$ in the setting of general
measure spaces $\left(X,\mathscr{B}\right)$. For each $R$, we compute
associated path-space probability spaces $\left(\Omega,\mathbb{P}\right)$.
When the transfer operator $R$ is compatible with an endomorphism
in $\left(X,\mathscr{B}\right)$, we get associated multiresolutions
for the Hilbert spaces $L^{2}\left(\Omega,\mathbb{P}\right)$ where
the path-space $\Omega$ may then be taken to be a solenoid. Our multiresolutions
include both orthogonality relations and self-similarity algorithms
for standard wavelets and for generalized wavelet-resolutions. Applications
are given to topological dynamics, ergodic theory, and spectral theory,
in general; to iterated function systems (IFSs), and to Markov chains
in particular.
\end{abstract}

\tableofcontents{}

\section{Introduction}

The purpose of our paper is two-fold, first (1) to make precise a
setting of general measure spaces, and families of positive transfer
operators $R$, and for each $R$ to compute the associated path-space
measures $\left(\Omega,\mathbb{P}\right)$; and secondly (2) to create
multiresolutions (Sections \ref{subsec:gmmr} and \ref{subsec:mul})
in the corresponding Hilbert spaces $L^{2}\left(\Omega,\mathbb{P}\right)$
of square integrable random variables.

We shall use the notion of ``transfer operator\textquotedblright{}
in a wide sense so that our framework will encompass diverse settings
from mathematics and its applications, including statistical mechanics
where the relevant operators are often referred to as Ruelle-operators
(Definitions \ref{def:ms} and \ref{def:gmR}; and we shall use the
notation $R$ for transfer operator for that reason.) See, e.g,. \cite{MR3124323,MR3459161,MR3375595,MR2176941,MR2129258}.
But we shall also consider families of transfer operators arising
in harmonic analysis, including spectral analysis of wavelets (\subsecref{gmwr}),
in ergodic theory of endomorphisms in measure spaces  (\remref{end}
and \secref{el}), in Markov random walk models, in the study of transition
processes in general; and more. 

In the setting of endomorphisms and solenoids, we obtain new multiresolution
orthogonality relations in the Hilbert space of square integrable
random variables. We shall further draw parallels between our present
infinite-dimensional theory and the classical finite-dimensional Perron-Frobenius
theorems (see, e.g., \cite{MR2176941,MR2129258,MR3506499,MR3375595,MR3368972,MR3461553});
the latter referring to the case of finite positive matrices.

To make this parallel, it is helpful to restrict the comparison of
the infinite-dimensional theory to the case of the Perron-Frobenius
(P-F) for finite matrices in the special case when the spectral radius
is 1.

Our present study of infinite-dimensional versions of P-F transfer
operators includes theorems which may be viewed as analogues of many
points from the classical finite-dimensional P-F case; for example,
the classical respective left and right Perron-Frobenius eigenvectors
now take the form in infinite-dimensions of positive $R$ invariant
measures (left), and the infinite-dimensional right P-F vector becomes
a positive harmonic function. Of course in infinite-dimensions, we
have more non-uniqueness than is implied by the classical matrix theorems,
but we also have many parallels. We even have infinite-dimensional
analogues of the P-F limit theorems from the classical matrix case.

Important points in our present consideration of transfer operators
are as follows: We formulate a general framework, a list of precise
axioms, which includes a diverse host of applications. In this, we
separate consideration of the transfer operators as they act on functions
on Borel spaces $\left(X,\mathscr{B}\right)$ on the one hand, and
their Hilbert space properties on the other hand. When a transfer
operator is given, there is a variety of measures compatible with
it, and we shall discuss both the individual cases, as well as the
way a given transfer operator is acting on a certain universal Hilbert
space (Definitions \ref{def:uh1} and \ref{def:uh2}). The latter
encompasses all possible probability measures on the given Borel space
$\left(X,\mathscr{B}\right)$. This yields new insight, and it helps
us organize our results on ergodic theoretic properties connected
to the theory of transfer operators, \secref{el}.

\section{\label{sec:ms}Measure spaces}

In the next two sections we make precise the setting of general measure
spaces, and families of positive transfer operators $R$, and we study
a number of convex sets of measures computed directly from $R$.

The general setting is as follows:
\begin{defn}
\label{def:ms}~
\begin{enumerate}
\item \label{enu:ms1}$\left(X,\mathscr{B}\right)$ is a fixed \emph{measure
space}, i.e., $\mathscr{B}$ is a fixed sigma-algebra of subsets of
a set $X$. Usually, we assume, in addition, that $\left(X,\mathscr{B}\right)$
is a Borel space.
\item \label{enu:ms2}Notation: $\sigma:X\rightarrow X$ is a measurable
\emph{endomorphism}, i.e., $\sigma^{-1}\left(\mathscr{B}\right)\subset\mathscr{B}$,
$\sigma^{-1}\left(A\right)\in\mathscr{B}$ for all $A\in\mathscr{B}$;
and we assume further that $\sigma\left(X\right)=X$, i.e., $\sigma$
is onto.
\item $\mathscr{F}\left(X,\mathscr{B}\right)$ = the algebra of all measurable
functions on $\left(X,\mathscr{B}\right)$.
\item By a \emph{transfer operator} $R$, we mean that $R:\mathscr{F}\left(X,\mathscr{B}\right)\longrightarrow\mathscr{F}\left(X,\mathscr{B}\right)$
is a linear operator s.t. (\ref{eq:as1}) $\&$ (\ref{eq:as2}) hold,
where:
\begin{equation}
f\geq0\Longrightarrow R\left(f\right)\geq0;\;\text{and}\label{eq:as1}
\end{equation}
\begin{equation}
R\left(\left(f\circ\sigma\right)g\right)=fR\left(g\right),\;\forall f,g\in\mathscr{F}\left(X,\mathscr{B}\right).\label{eq:as2}
\end{equation}
(See, e.g., \cite{MR3124323,MR3459161,MR3375595,MR2176941,MR2129258}.)
\item We assume that 
\begin{equation}
R\mathbbm{1}=\mathbbm{1}\label{eq:as4}
\end{equation}
where $\mathbbm{1}$ denotes the constant function ``one'' on $X$,
and we restrict consideration to the case of \emph{real} valued functions.
Subsequently, condition (\ref{eq:as4}) will be relaxed. 
\item \label{enu:ms6}If $\lambda$ is a measure on $\left(X,\mathscr{B}\right)$,
we set $\lambda R$ to be the measure specified by 
\begin{equation}
\int_{X}f\,d\left(\lambda R\right):=\int_{X}R\left(f\right)d\lambda,\;\forall f\in\mathscr{F}\left(X,\mathscr{B}\right).\label{eq:as3}
\end{equation}
\item We shall assume separability, for example we assume that $\left(X,\mathscr{B},\lambda\right)$,
as per (\ref{enu:ms1})\textendash (\ref{enu:ms6}), has the property
that $L^{2}\left(X,\mathscr{B},\lambda\right)$ is a separable Hilbert
space.
\end{enumerate}
\end{defn}
\begin{rem}
\label{rem:end}The role of the endomorphism $X\xrightarrow{\;\sigma\;}X$
is fourfold: 
\begin{enumerate}[label=(\alph{enumi})]
\item $\sigma$ is a point-transformation, generally not invertible, but
assumed \emph{onto}. 
\item We also consider $\sigma$ as an endomorphism in the fixed measure
space $\left(X,\mathscr{B}\right)$ and so $\sigma^{-1}:\mathscr{B}\rightarrow\mathscr{B}$
where
\begin{align*}
\sigma^{-1}\left(\mathscr{B}\right) & =\left\{ \sigma{}^{-1}\left(A\right)\mid A\in\mathscr{B}\right\} ,\:\text{and}\\
\sigma^{-1}\left(A\right) & :=\left\{ x\in X\mid\sigma\left(x\right)\in A\right\} ,
\end{align*}
so $\sigma^{-1}\left(\mathscr{B}\right)\subset\mathscr{B}$. 
\item We shall assume further that $\sigma$ is ergodic \cite{MR617913,MR3472541},
i.e., that
\[
\bigcap_{n=1}^{\infty}\sigma^{-n}\left(\mathscr{B}\right)=\left\{ \emptyset,X\right\} 
\]
modulo sets of $\lambda$-measure zero.
\item $\sigma$ defines an endomorphism in the space $\mathscr{F}\left(X,\mathscr{B}\right)$
of all measurable functions via $f\mapsto f\circ\sigma$. 
\end{enumerate}
\end{rem}

\section{\label{sec:sm}Sets of measures for $\left(X,\mathscr{B},\sigma,R\right)$ }

We shall undertake our analysis of particular transfer operators/endomorphisms
in a fixed measure space $\left(X,\mathscr{B}\right)$ with the use
of certain sets of measures on $\left(X,\mathscr{B}\right)$. These
sets play a role in our theorems, and they are introduced below. We
present examples of transfer operators associated to iterated function
systems (IFSs) in a stochastic framework. \exaref{eG} and \thmref{sm1}
prepare the ground for this, and the theme is resumed systematically
in \subsecref{ifsm} below.

For positive measures $\lambda$ and $\mu$ on $\left(X,\mathscr{B}\right)$,
we shall work with absolute continuity, written $\lambda\ll\mu$. 
\begin{defn}
$\lambda\ll\mu$ iff (Def.) {[}$A\in\mathscr{B}$, $\mu\left(A\right)=0$
$\Longrightarrow$ $\lambda\left(A\right)=0${]}. Moreover, when $\lambda\ll\mu$,
we denote the Radon-Nikodym derivative $\frac{d\lambda}{d\mu}$. In
detail, 
\[
\int_{B}\left(\frac{d\lambda}{d\mu}\right)d\mu=\lambda\left(B\right),\;B\in\mathscr{B}.
\]
Note that $\frac{d\lambda}{d\mu}\in L^{1}\left(\mu\right)$. 
\end{defn}
\begin{defn}
\label{def:end}Let $\sigma$ be an endomorphism in the measure space
$\left(X,\mathscr{B}_{X}\right)$, assuming $\sigma$ is onto. Introduce
the corresponding \emph{solenoid} 
\begin{equation}
Sol_{\sigma}\left(X\right):=\left\{ \left(x_{n}\right)_{0}^{\infty}\in\prod_{0}^{\infty}X\mid\sigma\circ\pi_{n+1}=\pi_{n}\right\} ;\label{eq:sig1}
\end{equation}
where $\pi_{n}\left(\left(x_{k}\right)\right):=x_{n}$, and we set
\begin{equation}
\tilde{\sigma}\left(x_{0},x_{1},x_{2}\cdots\right):=\left(\sigma\left(x_{0}\right),x_{0},x_{1},x_{2},\cdots\right),\;\forall x=\left(x_{i}\right)_{0}^{\infty}\in Sol_{\sigma}\left(X\right).\label{eq:sig2}
\end{equation}
\end{defn}
\begin{example}
\label{exa:eG}The following considerations cover an important class
of transfer operators which arise naturally in the study of controlled
Markov-processes, and in analysis of iterated function system (IFS),
see, e.g., \cite{MR544839,MR3343913,MR3203397} and \cite{MR1669737}.

Let $\left(X,\mathscr{B}_{X}\right)$ and $\left(Y,\mathscr{B}_{Y}\right)$
be two measure spaces. We equip $Z:=X\times Y$ with the product sigma-algebra
induced from $\mathscr{B}_{X}\times\mathscr{B}_{Y}$, and we consider
a fixed measurable function $G:Z\rightarrow X$. For $\nu\in M\left(Y,\mathscr{B}_{Y}\right)$
(= positive measures on $Y$), we set 
\begin{equation}
\left(Rf\right)\left(x\right)=\int_{Y}f\left(G\left(x,y\right)\right)d\nu\left(y\right),\label{eq:cm1}
\end{equation}
defined for all $f\in\mathscr{F}\left(X,\mathscr{B}_{X}\right)$.
This operator $R$ from (\ref{eq:cm1}) is a transfer operator; it
naturally depends on $G$ and $\nu$. 

If $\nu\in M_{1}\left(Y,\mathscr{B}_{Y}\right)$ (= the probability
measures), then $R\mathbbm{1}=\mathbbm{1}$, where $\mathbbm{1}$
denotes the constant function ``one'' on $X$. 

For every $x\in X$, $G\left(x,\cdot\right)$ is a measurable function
from $Y$ to $X$, which we shall denote $G_{x}$. It follows from
(\ref{eq:cm1}) that the marginal measures $\mu\left(\cdot\mid x\right)$
from the representation 
\begin{equation}
\left(Rf\right)\left(x\right)=\int_{X}f\left(t\right)\mu\left(dt\mid x\right)\label{eq:cm2}
\end{equation}
may be expressed as 
\begin{equation}
\mu\left(\cdot\mid x\right)=\nu\circ G_{x}^{-1},\label{eq:cm3}
\end{equation}
pull-back from $\nu$ via $G_{x}$. 
\end{example}
Set $M_{1}\left(X,\mathscr{B}\right):=$ all probability measures
on $\left(X,\mathscr{B}\right)$, and 
\[
\mathscr{L}_{1}\left(R\right):=\left\{ \lambda\in M_{1}\left(X,\mathscr{B}\right)\mid\lambda R=\lambda\right\} 
\]
where $\int_{X}f\,d\left(\lambda R\right):=\int_{X}R\left(f\right)d\lambda$,
$\forall f$. 

The following lemma is now immediate.
\begin{lem}
\label{lem:cm}Let $G$, $\nu$, and $R$ be as above, with $R$ given
by (\ref{eq:cm1}), or equivalently by (\ref{eq:cm2}); then a fixed
measure $\lambda$ on $\left(X,\mathscr{B}_{X}\right)$ is in $\mathscr{L}_{1}\left(R\right)$
iff
\begin{equation}
\lambda\left(B\right)=\int_{X}\nu\left(\left\{ y\,:\,G\left(x,y\right)\in B\right\} \right)d\lambda\left(x\right)\label{eq:cm4}
\end{equation}
for all $B\in\mathscr{B}_{X}$.
\end{lem}
\begin{proof}
Immediate from the definitions. 
\end{proof}

\begin{rem}
(a) The reader will be able to write formulas for the other sets in
\defref{meas}, analogous to (\ref{eq:cm4})\@. 

(b) The conditions in the discussion of \lemref{cm} apply to the
following example.
\end{rem}
\begin{prop}
\label{prop:me}Let $X=\left(0,1\right)$ = the open unit interval
with the standard Borel sigma-algebra, and let $Y=\left(0,1\right)\times\left\{ 0,1\right\} $
with measure $\nu$ on $Y$: 
\begin{align}
\nu & =\left(\text{Lebesgue}\right)\times\left(\text{fair coin}\right)\label{eq:me1}\\
 & =\left(du\right)\times\left(\frac{1}{2},\frac{1}{2}\right).\nonumber 
\end{align}
Set $G:X\times Y\rightarrow X$ by (\figref{me})
\begin{equation}
\left.\begin{split}G\left(x,\left(u,0\right)\right) & =ux & \text{if }i=0\\
G\left(x,\left(u,1\right)\right) & =\left(1-u\right)x+u & \text{if }i=1
\end{split}
\right\} .\label{eq:me2}
\end{equation}

Then we have
\begin{equation}
\left(Rf\right)\left(x\right)=\frac{1}{2}\left(\frac{1}{x}\int_{0}^{x}f\left(t\right)dt+\frac{1}{1-x}\int_{x}^{1}f\left(t\right)dt\right)\label{eq:me3}
\end{equation}
with transpose 
\begin{equation}
f\longmapsto\frac{1}{2}\left(\int_{y}^{1}\frac{f\left(x\right)}{x}dx+\int_{0}^{y}\frac{f\left(x\right)}{1-x}dx\right)\label{eq:me4}
\end{equation}
and 
\begin{equation}
d\lambda\left(x\right)=\frac{dx}{\pi\sqrt{x\left(1-x\right)}}\label{eq:me5}
\end{equation}
satisfying $\lambda R=\lambda$, i.e., $\lambda\in\mathscr{L}_{1}\left(R\right)$. 
\end{prop}

\begin{proof}
(sketch) Direct verification: Note that if $d\lambda=g\left(x\right)dx$
satisfies $\lambda R=\lambda$ then by (\ref{eq:me4}), we have 
\begin{equation}
\frac{g'\left(y\right)}{g\left(y\right)}=\frac{1}{2}\left(-\frac{1}{y}+\frac{1}{1-y}\right),\label{eq:me6}
\end{equation}
and the result follows.
\end{proof}
\begin{figure}
\subfloat[$i=0$]{\includegraphics[width=0.2\textwidth]{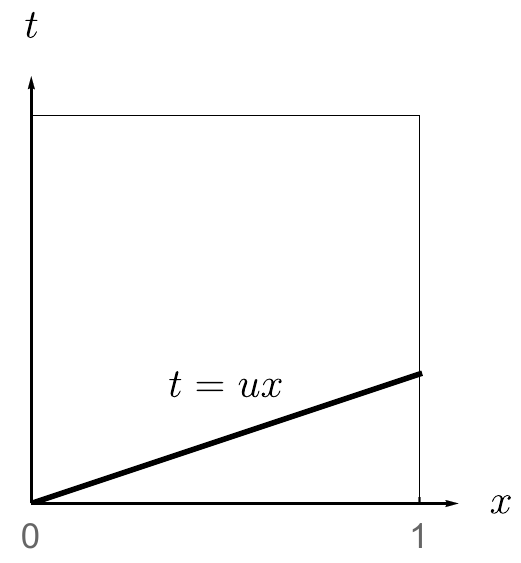}

}\qquad{}\subfloat[$i=1$]{\includegraphics[width=0.2\textwidth]{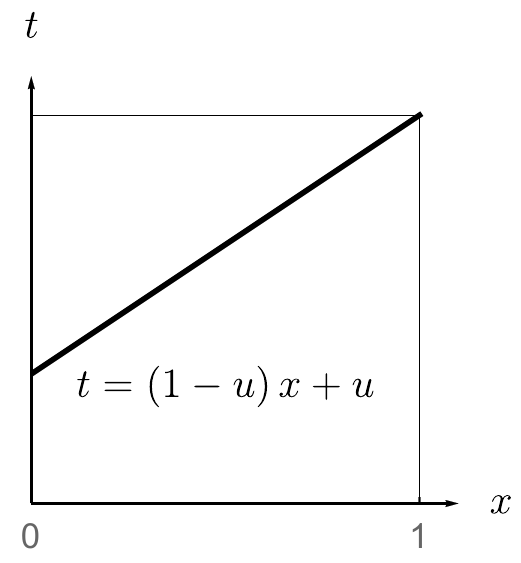}

}

\caption{\label{fig:me}The function $G$, see (\ref{eq:me2}).}
\end{figure}

\begin{rem}[Reflection symmetry]
Let $R$ be as in (\ref{eq:me3}) and $\lambda$ given by (\ref{eq:me5}).
Set $\sigma\left(x\right)=1-x$. Then the following reflection symmetry
holds:
\[
R\left(f\circ\sigma\right)=R\left(f\right)\circ\sigma,\;\forall f\in\mathscr{F}\left(X,\mathscr{B}\right),
\]
and $\lambda\circ\sigma^{-1}=\lambda$. 
\end{rem}

The purpose of the next theorem is to make precise the direct connections
between the following three notions, a given positive transfer operator,
an induced probability space, and an associated Markov chain \cite{MR3514301,MR3508495}.
\begin{thm}
\label{thm:sm1}Fix $h\geq0$ on $\left(X,\mathscr{B}_{X}\right)$
s.t. $Rh=h$, and $\int_{X}h\,d\lambda=1$. 
\begin{enumerate}
\item Then $\Omega_{X}:=\prod_{0}^{\infty}X$ supports a probability space
$\left(\Omega_{X},\mathscr{F},\mathbb{P}\right)$ (\defref{gm1}),
such that $\mathbb{P}$ is determined by the following: 
\begin{align}
 & \int_{\Omega_{X}}\left(f_{0}\circ\pi_{0}\right)\left(f_{1}\circ\pi_{1}\right)\cdots\left(f_{n}\circ\pi_{n}\right)d\mathbb{P}\nonumber \\
= & \int_{X}f_{0}\left(x\right)R\left(f_{1}R\left(f_{2}\cdots R\left(f_{n}h\right)\right)\cdots\right)\left(x\right)d\lambda\left(x\right),\label{eq:pmm1}
\end{align}
where $\pi_{n}$ is the coordinate mapping in (\ref{eq:gm1}), $\pi_{n}\left(\left(x_{i}\right)\right)=x_{n}$. 

More generally, 
\begin{align}
 & \text{Prob}\left(\pi_{0}=x,\pi_{1}\in B_{1},\pi\in B_{2},\cdots,\pi_{n}\in B_{n}\right)\nonumber \\
= & \int_{B_{1}}\int_{B_{2}}\cdots\int_{B_{n}}\mu\left(dy_{1}\mid x\right)\mu\left(dy_{2}\mid y_{1}\right)\cdots\mu\left(dy_{n}\mid y_{n-1}\right)h\left(y_{n}\right)\nonumber \\
= & R\left(\chi_{B_{1}}R\left(\chi_{B_{2}}\cdots R\left(\chi_{B_{n}}h\right)\right)\cdots\right)\left(x\right),\;\forall B_{j}\in\mathscr{B}_{X}.\label{eq:pmm2}
\end{align}

\item If $d\left(\lambda R\right)=Wd\lambda$, then 
\begin{equation}
\mathbb{P}\circ\pi_{1}^{-1}=\left(\left(W\circ\pi_{0}\right)d\mathbb{P}\right)\circ\pi_{0}^{-1}.
\end{equation}
\item Moreover, 
\begin{align}
\text{suppt}\left(\mathbb{P}\right) & =Sol_{\sigma}\left(X\right)\nonumber \\
 & \Updownarrow\\
R\left[\left(f\circ\sigma\right)g\right] & =fR\left(g\right),\;\forall f,g\in\mathscr{F}\left(X,\mathscr{B}\right).\nonumber 
\end{align}
\end{enumerate}
\end{thm}
\begin{proof}
Follows from Kolmogorov's inductive limit construction. For details,
see \cite{JT15a1,MR3275999,MR2726223,MR2461126,MR2360935} and also
\cite{MR562914,MR3272038,MR0279844}.
\end{proof}
\begin{rem}
When we pass from $\left(X,\mathscr{B},R,h,\lambda\right)$ to the
corresponding $L^{2}\left(\Omega_{X},\mathscr{C},\mathbb{P}\right)$
as in \thmref{sm1}, then the sigma-algebras $\sigma^{-n}\left(\mathscr{B}\right)$
induce a filtration also for the sigma-algebra $\mathscr{C}$ of cylinder
sets in $\Omega_{X}$. Here $\mathscr{C}$ denotes the sigma-algebra
of subsets in $\Omega_{X}$ generated by $\left\{ \pi_{n}^{-1}\left(\mathscr{B}\right)\mid n\in\mathbb{Z}_{+}\cup\left\{ 0\right\} \right\} $. 
\end{rem}
\begin{defn}
A subset $L\subset M_{1}$ is said to be \emph{closed} iff it is closed
in the $w^{*}$-topology on $M_{1}$, i.e., the topology defined by
the bilinear pairing
\begin{equation}
\left(\lambda,f\right)\longmapsto\int_{X}f\,d\lambda,\quad\lambda\in M_{1},f\in\mathscr{F}\left(X,\mathscr{B}\right).\label{eq:a1}
\end{equation}
\end{defn}
\begin{defn}
\label{def:meas}Set 
\begin{align}
\mathscr{L}\left(R\right) & :=\left\{ \lambda\in M_{1}\mid\lambda R\ll\lambda\right\} ;\label{eq:a2}\\
\mathscr{K}_{1} & :=\left\{ \lambda\in M_{1}\mid\left(\lambda\circ\sigma^{-1}\right)R=\lambda\right\} ;\label{eq:a3}\\
Fix\left(\sigma\right) & :=\left\{ \lambda\in M_{1}\mid\lambda\circ\sigma^{-1}=\lambda\right\} ;\;\mbox{and}\label{eq:a4}\\
\mathscr{L}_{1}\left(R\right) & :=\left\{ \lambda\in\mathscr{L}\left(R\right)\mid\lambda R=\lambda\right\} .\label{eq:a5}
\end{align}
\end{defn}
\begin{lem}
The sets in (\ref{eq:a2})-(\ref{eq:a5}) are convex and closed.
\end{lem}
\begin{proof}
The first part is easy, and the second part follows from the following
considerations. For the cases (\ref{eq:a3})-(\ref{eq:a5}), we use
the pairing (\ref{eq:a1}):
\begin{align*}
\int_{X}\left(R\left(f\right)\circ\sigma\right)d\lambda & =\int_{X}f\,d\left(\lambda\circ\sigma^{-1}\right)R,\\
\int_{X}f\circ\sigma\,d\lambda & =\int_{X}f\,d\left(\lambda\circ\sigma^{-1}\right),\;\mbox{and}\\
\int_{X}R\left(f\right)d\lambda & =\int_{X}f\,d\left(\lambda R\right),
\end{align*}
for $\forall f\in\mathscr{F}\left(X,\mathscr{B}\right)$, $\lambda\in M_{1}$.

The proof that $\mathscr{L}\left(R\right)$ in (\ref{eq:a2}) is $w^{*}$-closed
uses the following symmetry: 
\begin{equation}
\int\left(f\circ\sigma\right)\left(\frac{d\lambda R}{d\lambda}\right)g\,d\lambda=\int f\,R\left(g\right)d\lambda,\label{eq:a6}
\end{equation}
$\forall f,g\in\mathscr{F}\left(X,\mathscr{B}\right)$, $\forall\lambda\in\mathscr{L}\left(R\right)$.
\end{proof}
\begin{lem}
\label{lem:m1}Let $\left(X,\mathscr{B},\sigma,R\right)$ be as specified.
Then TFAE: 
\begin{enumerate}
\item \label{enu:ml1}$\lambda\in\mathscr{K}_{1}$ $\left(\text{i.e., }\left(\lambda\circ\sigma^{-1}\right)R=\lambda\right)$;
\item \label{enu:ml2}$\lambda\in M_{1}R\:\left(=:\left\{ \nu R\mid\nu\in M_{1}\right\} \right)$; 
\item \label{enu:ml3}The mapping
\[
f\longmapsto R\left(f\right)\circ\sigma\bigm|_{L^{2}\left(\lambda\right)}=\mathbb{E}^{\left(\lambda\right)}\left(f\mid\sigma^{-1}\left(\mathscr{B}\right)\right)
\]
 is the $\lambda$-$\sigma^{-1}\left(\mathscr{B}\right)$ conditional
expectation (\defref{gmce}).
\end{enumerate}
\end{lem}
\begin{proof}
~

(\ref{enu:ml1}) $\Longleftrightarrow$ (\ref{enu:ml2}). Immediate
from the definitions. 

(\ref{enu:ml2}) $\Longrightarrow$ (\ref{enu:ml3}). It is clear
that LHS in (\ref{enu:ml3}) has the properties of conditional expectation
(as stated) except for the Hermitian property; i.e., 
\begin{equation}
\int_{X}\left(R\left(f_{1}\right)\circ\sigma\right)f_{2}d\lambda=\int_{X}f_{1}\left(R\left(f_{2}\right)\circ\sigma\right)d\lambda,\;\forall f_{1},f_{2}\in\mathscr{F}\left(X,\mathscr{B}\right).\label{eq:a7}
\end{equation}
To prove (\ref{eq:a7}), we use (\ref{enu:ml2}), i.e., that there
is a $\nu\in M_{1}$ s.t. $\lambda=\nu R$. Then we get:
\begin{align*}
\mbox{LHS}_{\left(\ref{eq:a7}\right)} & =\int_{X}\left(R\left(f_{1}\right)\circ\sigma\right)f_{2}d\left(\nu R\right)\\
 & =\int_{X}R\left[\left(R\left(f_{1}\right)\circ\sigma\right)f_{2}\right]d\nu\\
 & =\int_{X}R\left(f_{1}\right)R\left(f_{2}\right)d\nu=\mbox{RHS}_{\left(\ref{eq:a7}\right)},\;\text{by symmetry.}
\end{align*}

(\ref{enu:ml3}) $\Longrightarrow$ (\ref{enu:ml1}). Set $f_{2}=\mathbbm{1}$
in (\ref{eq:a7}), and use the assumption $R\left(\mathbbm{1}\right)=\mathbbm{1}$.

In order to show that the operator $Q$ in (\ref{enu:ml3}) is the
stated conditional expectation, we must verify the following
\begin{enumerate}[label=(\roman{enumi})]
\item $Q\left(f\circ\sigma\right)=f\circ\sigma$, $\forall f\in\mathscr{F}\left(X,\mathscr{B}\right)$; 
\item $Q^{2}=Q=Q^{*}$, where the adjoint $Q^{*}$ refers to $L^{2}\left(X,\mathscr{B},\lambda\right)$. 
\end{enumerate}
Proof of (i). On $L^{2}\left(X,\mathscr{B},\lambda\right)$ we have
the following: 
\begin{align*}
Q\left(f\circ\sigma\right) & =R\left(f\circ\sigma\right)\circ\sigma\\
 & =\left(fR\left(\mathbbm{1}\right)\right)\circ\sigma=f\circ\sigma,
\end{align*}
which is the desired conclusion. 

Proof of (ii). The same argument proves that $Q^{2}=Q$, so we turn
to $Q^{*}=Q$, which is (\ref{eq:a7}) above. Note that once (i)\textendash (ii)
are established, then it is clear that 
\begin{equation}
\int_{X}\left(f_{1}\circ\sigma\right)\left(Qf_{2}\right)d\lambda=\int_{X}\left(f_{1}\circ\sigma\right)f_{2}\,d\lambda,\;\forall f_{1},f_{2}\in\mathscr{F}\left(X,\mathscr{B}\right);\label{eq:a7a}
\end{equation}
since, using $Q^{*}=Q$, 
\[
\text{LHS}_{\left(\ref{eq:a7a}\right)}=\int_{X}Q\left(f_{1}\circ\sigma\right)f_{2}\,d\lambda\underset{\text{by (i)}}{=}\int_{X}\left(f_{1}\circ\sigma\right)f_{2}\,d\lambda=\text{RHS}_{\left(\ref{eq:a7a}\right)}.
\]
\end{proof}
\begin{cor}
Let $\left(X,\mathscr{B}\right)$ be a measure space, and $R$ a positive
operator s.t. $\exists\lambda\in M_{1}\left(X,\mathscr{B}\right)$
(= probability measures) with 
\begin{equation}
\lambda R=\lambda,\quad R\mathbbm{1}=\mathbbm{1}.\label{eq:b1}
\end{equation}
Suppose an endomorphism $\sigma$ in $\left(X,\mathscr{B}\right)$
mapping onto $X$ exists satisfying 
\begin{equation}
\lambda\circ\sigma^{-1}=\lambda.\label{eq:b2}
\end{equation}
Assume further 
\begin{equation}
\int_{X}R\left(f\right)g\,d\lambda=\int_{X}f\left(g\circ\sigma\right)d\lambda,\;\forall f,g\in\mathscr{F}\left(X,\mathscr{B}\right).\label{eq:b3}
\end{equation}

Then 
\begin{equation}
R\left(\left(f\circ\sigma\right)g\right)=fR\left(g\right),\;\forall f,g\in\mathscr{F}\left(X,\mathscr{B}\right)\label{eq:b4}
\end{equation}
holds if and only if 
\[
f\longmapsto R\left(f\right)\circ\sigma\big|_{L^{2}\left(X,\lambda\right)}
\]
is the conditional expectation $\mathbb{E}\left(f\mid\sigma^{-1}\left(\mathscr{B}\right)\right)$
in \lemref{m1}.
\end{cor}
\begin{proof}
The ``only if'' part is contained in \lemref{m1}.

For the ``if'' part, assume $\sigma$, $\lambda$, $R$ satisfy
the stated conditions, in particular that 
\[
R\left(f\right)\circ\sigma=\mathbb{E}\left(f\mid\sigma^{-1}\left(\mathscr{B}\right)\right),\;\forall f\in L^{2}\left(X,\lambda\right).
\]
Let $f,g\in L^{2}\left(X,\lambda\right)$, and $k\in L^{\infty}\left(X,\lambda\right)$.
Then 
\begin{align*}
 & \int_{X}R\left[\left(f\circ\sigma\right)g\right]k\,d\lambda\\
= & \int_{X}\left(f\circ\sigma\right)g\left(k\circ\sigma\right)d\lambda,\;\text{by \ensuremath{\left(\ref{eq:b3}\right)}}\\
= & \int_{X}\left(f\circ\sigma\right)\left(R\left(g\right)\circ\sigma\right)\left(k\circ\sigma\right)d\lambda,\;\text{the conditional expectation property}\\
= & \int_{X}\left(fR\left(g\right)k\right)\circ\sigma\,d\lambda\\
= & \int_{X}fR\left(g\right)k\,d\lambda,\;\text{by \ensuremath{\left(\ref{eq:b2}\right)}}.
\end{align*}
Since this holds when $f$ and $g$ are fixed, for $\forall k\in L^{\infty}\left(X,\lambda\right)$,
it follows that (\ref{eq:b4}) is satisfied. 
\end{proof}

\begin{rem}
The example from \propref{me} shows that there are positive transfer
operators $R$, $\lambda\in M_{1}\left(X,\mathscr{B}\right)$, with
$\lambda R=\lambda$, but such that 
\begin{equation}
R\left(\left(f\circ\sigma\right)g\right)=fR\left(g\right),\;f,g\in\mathscr{F}\left(X,B\right)\label{eq:b5}
\end{equation}
is \emph{not} satisfied for any endomorphism $\sigma$. 

Indeed, let $R$ be as in (\ref{eq:me3}) and assume (\ref{eq:b5})
holds. Then with $g=\mathbbm{1}$ and $f\left(x\right)=x^{n}$, we
must have 
\[
x^{n}=\frac{1}{2}\left(\frac{1}{x}\int_{0}^{x}\left(\sigma\left(t\right)\right)^{n}dt+\frac{1}{1-x}\int_{x}^{1}\left(\sigma\left(t\right)\right)^{n}dt\right),\;\forall n\in\mathbb{N}\cup\left\{ 0\right\} .
\]
Setting $x=\frac{1}{2}$, it follows that $\int_{0}^{1}\left(2\sigma\left(t\right)\right)^{n}dt=1$,
$\forall n$; and so $\sigma\equiv1/2$ a.e. But this is clearly a
contradiction. (The conclusion also follows from \thmref{Im} below.)
\end{rem}
We now turn to the general setting when a non-trivial endomorphism
$\sigma$ exists such that the compatibility (\ref{eq:b5}) is satisfied. 

We shall need the following:
\begin{lem}
\label{lem:abs1}The following implication holds:
\begin{equation}
\lambda\ll\mu\Longrightarrow\lambda R\ll\mu R,\label{eq:a7-1}
\end{equation}
and 
\begin{equation}
\frac{d\left(\lambda R\right)}{d\left(\mu R\right)}=\left(\frac{d\lambda}{d\mu}\right)\circ\sigma.\label{eq:a8}
\end{equation}
\end{lem}
\begin{proof}
Assume $\lambda\ll\mu$, and let $W=d\lambda/d\mu$ = the Radon-Nikodym
derivative. 

Then for $f\in\mathscr{F}\left(X,\mathscr{B}\right)$, we have:
\begin{align*}
\int_{X}f\,d\left(\lambda R\right) & =\int_{X}R\left(f\right)d\lambda=\int_{X}R\left(f\right)W\,d\mu\\
 & =\int_{X}R\left(\left(W\circ\sigma\right)f\right)d\mu=\int_{X}f\,\left(W\circ\sigma\right)\,d\left(\mu R\right),
\end{align*}
and the desired conclusion (\ref{eq:a8}) follows.
\end{proof}
In the theorem below we state our first result regarding the sets
of measures from \defref{meas}. The theorem will be used in Sections
\ref{subsec:mul} and \ref{sec:me} in our study of multiresolutions.
\begin{thm}
\label{thm:meas1}Let $\left(X,\mathscr{B},\sigma,R\right)$ be as
specified, and suppose that $R\left(\mathbbm{1}\right)=\mathbbm{1}$.
Let the sets of measures $\mathscr{L}\left(R\right)$, $\mathscr{K}_{1}$,
$Fix\left(\sigma\right)$, and $\mathscr{L}_{1}\left(R\right)$ be
as stated in \defref{meas}. Then
\begin{enumerate}
\item \label{enu:mm1}$Fix\left(\sigma\right)\cap\mathscr{K}_{1}=\mathscr{L}_{1}\left(R\right)$,
and 
\item \label{enu:mm2}$\mathscr{L}_{1}\left(R\right)\subset\mathscr{L}\left(R\right)\subset\mathscr{L}\left(R^{2}\right)\subset\cdots$
\end{enumerate}
\end{thm}
\begin{proof}
~

Part (\ref{enu:mm1}). Let $\lambda\in Fix\left(\sigma\right)\cap\mathscr{K}_{1}$,
then $\lambda=\left(\lambda\circ\sigma^{-1}\right)R=\lambda R$, and
so $\lambda\in\mathscr{L}_{1}\left(R\right)$. Conversely, suppose
$\lambda R=\lambda$, then $\lambda\in\mathscr{K}_{1}$ by \lemref{m1}.
On the other hand, since 
\[
\left(\lambda R\right)\circ\sigma^{-1}\underset{\text{by }\left(\ref{eq:as2}\right)}{=}\lambda=\lambda\circ\sigma^{-1},
\]
we get $\lambda\in Fix\left(\sigma\right)$.  

Part (\ref{enu:mm2}). Let $\lambda\in\mathscr{L}\left(R\right)$,
and set $Q:=d\left(\lambda R\right)/d\lambda$, i.e., 
\begin{equation}
\int_{X}R\left(f\right)d\lambda=\int_{X}fQ\,d\lambda,\quad\forall f\in\mathscr{F}\left(X,\mathscr{B}\right).\label{eq:a9}
\end{equation}
Then 
\begin{align*}
\int_{X}R^{2}\left(f\right)d\lambda & =\int_{X}R\left(R\left(f\right)\right)d\lambda=\int_{X}R\left(f\right)Q\,d\lambda\\
 & =\int_{X}R\left[f\,\left(Q\circ\sigma\right)\right]d\lambda=\int_{X}f\,\left(Q\circ\sigma\right)Q\,d\lambda,
\end{align*}
and so $\lambda R^{2}\ll\lambda$ with the Radon-Nikodym derivative
\begin{equation}
\frac{d\left(\lambda R^{2}\right)}{d\lambda}=\left(Q\circ\sigma\right)Q.\label{eq:a10}
\end{equation}
By induction, $\lambda R^{n}\ll\lambda$, with 
\begin{equation}
\frac{d\left(\lambda R^{n}\right)}{d\lambda}=\prod_{k=0}^{n-1}\left(Q\circ\sigma^{k}\right),\quad n=1,2,3\cdots.\label{eq:a11}
\end{equation}
Part (\ref{enu:mm2}) of the theorem follows from this.
\end{proof}

\section{IFSs in the measurable category}

We study here transfer operators associated to iterated function systems
(IFSs) in a stochastic framework. We begin with the traditional setting
(\subsecref{ifst}) as it will be part of the construction of the
generalized stochastic IFSs (\subsecref{ifsm}).

\subsection{\label{subsec:ifst}IFSs: Traditional}
\begin{defn}
Let $\left(X,\mathscr{B}\right)$ be a measure space and let $J$
be a countable index set. A system of endomorphisms $\left\{ \tau_{j}\right\} _{j\in J}$
in $\left(X,\mathscr{B}\right)$ is called an \emph{iterated function
system} (IFS) iff for all weights $p_{j}>0$ s.t. $\sum_{j}p_{j}=1$,
there is a probability measure $\mu$ on $\left(X,\mathscr{B}\right)$
satisfying
\begin{equation}
\sum_{j}p_{j}\int_{X}f\circ\tau_{j}\,d\mu=\int_{X}f\,d\mu,\;\forall f\in\mathscr{F}\left(X,\mathscr{B}\right);\label{eq:I1}
\end{equation}
or equivalently, 
\begin{equation}
\sum_{j}p_{j}\,\mu\circ\tau_{j}^{-1}=\mu.\label{I2}
\end{equation}
We say that $\mu$ is a $\left(p_{i}\right)$-\emph{equilibrium measure}
for the IFS. 
\end{defn}
When additional metric assumptions are placed on $(X,\mathscr{B},\left\{ \tau_{j}\right\} _{j\in J})$,
the existence (and possible uniqueness) of equilibrium measures $\mu$
have been studied; see, e.g., \cite{MR625600,MR1669737,MR1711343,MR3375595,MR2129258}.
\begin{example}
\label{exa:Iu}When $u\in\left(0,1\right)$ in (\ref{eq:me2}) from
\propref{me} is fixed, we get an IFS with $J=\left\{ 0,1\right\} $
as follows:
\begin{equation}
\begin{split}\tau_{0}^{\left(u\right)}\left(x\right) & =ux\\
\tau_{1}^{\left(u\right)}\left(x\right) & =\left(1-u\right)x+u,\;x\in\left(0,1\right)
\end{split}
\label{eq:I3}
\end{equation}
and the endomorphism (see \figref{su})
\begin{equation}
\sigma^{\left(u\right)}\left(x\right)=\begin{cases}
\dfrac{x}{u} & 0<x\leq u\\
\dfrac{x-u}{1-u} & u<x<1
\end{cases}\label{eq:I4}
\end{equation}
satisfying 
\begin{equation}
\sigma^{\left(u\right)}\circ\tau_{j}^{\left(u\right)}=id,\;j=0,1.\label{eq:I5}
\end{equation}

It further follows from \cite{MR625600} that for every $u\in\left(0,1\right)$,
fixed, there is a unique probability measure $\mu^{\left(u\right)}$
on $0<x<1$ such that 
\begin{equation}
\frac{1}{2}\int_{0}^{1}\left(f\left(ux\right)+f\left(\left(1-u\right)x+u\right)\right)d\mu^{\left(u\right)}\left(x\right)=\int_{0}^{1}f\,d\mu^{\left(u\right)}.\label{eq:I51}
\end{equation}
If $u<\frac{1}{2}$, these measures are singular and mutually singular;
i.e., if $u$ and $u'$ are different, the corresponding measures
are mutually singular. Moreover, if $u=\frac{1}{2}$, i.e., the measure
$\mu^{\left(\frac{1}{2}\right)}$, is the restriction of Lebesgue
measure to $0<x<1$. Nonetheless, when $R$ is as in (\ref{eq:me3})
from \propref{me}, then the unique probability measure satisfying
$\lambda R=\lambda$ is absolutely continuous, since $d\lambda\left(x\right)=\frac{dx}{\pi\sqrt{x\left(1-x\right)}}$
(see (\ref{eq:me5})). 

The measures $\mu^{\left(u\right)}$, for $u<\frac{1}{2}$, are examples
of fractal measures which are determined by affine self-similarity
\cite{MR3432848}, and, for $u$ fixed, $\mu^{\left(u\right)}$ has
scaling dimension $D\left(u\right)=-\ln2/\ln u$. These measures serve
as models for scaling-symmetry in a number of applications; see e.g.,
\cite{MR625600} and \cite{MR1426612,MR879247}.
\end{example}
\begin{figure}
\includegraphics[width=0.2\textwidth]{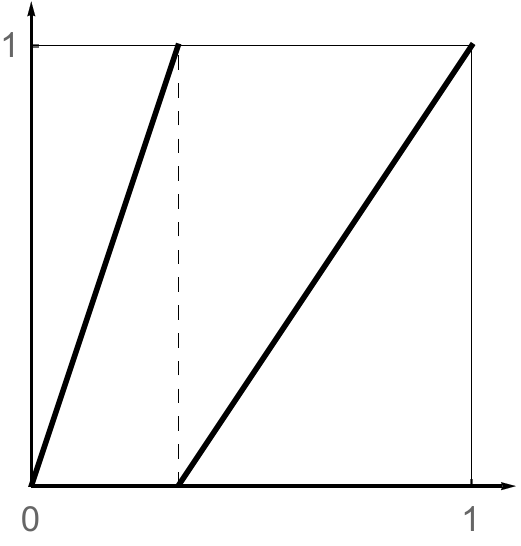}

\caption{\label{fig:su}The endomorphism $\sigma^{\left(u\right)}$ from (\ref{eq:I4}). }

\end{figure}

\begin{defn}
\label{def:Istable}An IFS $\left\{ \tau_{j}\right\} _{j\in J}$ in
$\left(X,\mathscr{B}\right)$, the given measure space, is said to
be \emph{stable} iff there is an endomorphism $\sigma$ in $\left(X,\mathscr{B}\right)$
such that 
\begin{equation}
\sigma\circ\tau_{j}=id_{X},\;\forall j\in J.\label{eq:I6}
\end{equation}
\end{defn}
\begin{rem}
Suppose $\left(X,\mathscr{B},\left\{ \tau_{j}\right\} ,\left\{ p_{j}\right\} \right)$
is a stable IFS; set
\begin{equation}
\left(Rf\right)\left(x\right)=\sum_{j}p_{j}f\left(\tau_{j}\left(x\right)\right),\;x\in X,\label{eq:I7}
\end{equation}
then this transfer operator $R$ satisfies 
\begin{equation}
R\left[\left(f\circ\sigma\right)g\right]=fR\left(g\right),\;\forall f,g\in\mathscr{F}\left(X,\mathscr{B}\right);\label{eq:I8}
\end{equation}
but in general (\ref{eq:I8}) may not be satisfied for any choice
of endomorphism $\sigma$.
\end{rem}

\subsection{\label{subsec:ifsm}IFSs: The measure category}

We now return to the setting 
\begin{equation}
G:X\times Y\longrightarrow X\label{eq:I9}
\end{equation}
from \exaref{eG} where $\left(X,\mathscr{B}_{X}\right)$ and $\left(Y,\mathscr{B}_{Y}\right)$
are given measure spaces, $G$ in (\ref{eq:I9}) is measurable from
$X\times Y$ to $X$, and $X\times Y$ is given the product sigma-algebra. 

We saw that for every choice of probability measure $\nu$ on $\left(Y,\mathscr{B}_{Y}\right)$,
we get a corresponding transfer operator (\ref{eq:cm1}), depending
on both $G$ and $\nu$. We further assume that $G\left(\cdot,y\right)$
is 1-1 on $X$, for $y\in Y$. 
\begin{thm}
\label{thm:Im}Let $G:X\times Y\rightarrow X$ be as in (\ref{eq:I9})
for given measure spaces $\left(X,\mathscr{B}_{X}\right)$ and $\left(Y,\mathscr{B}_{Y}\right)$,
let $\nu\in M_{1}\left(Y,\mathscr{B}_{Y}\right)$, and $\lambda\in M_{1}\left(X,\mathscr{B}_{X}\right)$
be fixed probability measures. Let $R=R_{\left(G,\nu\right)}$ be
the corresponding transfer operator in $L^{2}\left(X,\lambda\right)$
given by 
\begin{equation}
\left(Rf\right)\left(x\right)=\int_{Y}f\left(G\left(x,y\right)\right)d\nu\left(y\right),\;f\in\mathscr{F}\left(X,\mathscr{B}_{X}\right).\label{eq:I10}
\end{equation}

A given endomorphism $\sigma$ in $\left(X,\mathscr{B}_{X}\right)$
satisfies 
\begin{equation}
R_{\left(G,\nu\right)}\left[\left(f_{1}\circ\sigma\right)f_{2}\right]=f_{1}R_{\left(G,\nu\right)}\left(f_{2}\right),\;\forall f_{1},f_{2}\in\mathscr{F}\left(X,\mathscr{B}_{X}\right)\label{eq:I11}
\end{equation}
if and only if 
\begin{equation}
\sigma\left(G\left(x,y\right)\right)=x,\label{eq:I12}
\end{equation}
a.e. $y$ w.r.t. $\nu$, and a.e. $x$ w.r.t. $\lambda$. 
\end{thm}
\begin{proof}
It is immediate that (\ref{eq:I12}) $\Longrightarrow$ (\ref{eq:I11}).
Conversely suppose (\ref{eq:I11}) holds. We then get 
\begin{align*}
 & \int_{Y}f_{1}\left(\sigma\left(G\left(x,y\right)\right)\right)f_{2}\left(G\left(x,y\right)\right)d\nu\left(y\right)\\
= & f_{1}\left(x\right)\int_{Y}f_{2}\left(G\left(x,y\right)\right)d\nu\left(y\right),
\end{align*}
$\forall f_{1},f_{2}\in\mathscr{F}\left(X,\mathscr{B}_{X}\right)$,
a.e. $x$ w.r.t. $\lambda$.

From the assumptions in the theorem, we conclude that the following
identity holds for measures
\[
\int\sigma\left(G\left(x,y\right)\right)d\nu\left(y\right)=\delta_{x},
\]
a.e. $x$ (w.r.t. $\lambda$), and therefore 
\[
\sigma\left(G\left(x,y\right)\right)=x,
\]
a.e. $y$ w.r.t. $\nu$, and a.e. $x$ w.r.t. $\lambda$, which is
the desired conclusion (\ref{eq:I12}).
\end{proof}

\begin{rem}
It is easy to see that if $G$ is as in (\ref{eq:me2}) in \propref{me},
then there is no solution $\sigma\in End\left(\left(0,1\right),\mathscr{B}\right)$
to the condition in (\ref{eq:I12}); and so by the theorem; this particular
IFS (in the generalized sense) is \emph{not} stable in the sense of
\defref{Istable}.
\end{rem}
\begin{defn}
Let $\left(X,\mathscr{B}_{X}\right)$, $\left(Y,\mathscr{B}_{Y}\right)$,
$G$, and $\nu$ be as in the statement of \thmref{Im}. Let $R=R_{\left(G,\nu\right)}$
be the corresponding transfer operator, see (\ref{eq:I10}). 

Suppose $Y$ has the following factorization, $Y=U\times J$, where
$\left(U,\mathscr{B}_{U}\right)$ is a measure space and $J$ is an
at most countable index set. Let $\nu\left(\cdot\mid i\right)$, $i\in J$,
be the induced conditional measures on $U$, i.e., for some $\left\{ p_{j}\right\} _{j\in J}$
we have 
\begin{equation}
\nu\left(\pi_{U}\in A,\:\pi_{J}=i\right)=p_{i}\nu\left(A\mid i\right)\label{eq:Ir1}
\end{equation}
for all $A\in\mathscr{B}_{U}$, $i\in J$, where 
\begin{equation}
\pi_{U}\left(\left(u,i\right)\right)=u,\;\text{and }\pi_{J}\left(\left(u,i\right)\right)=i.\label{eq:Ir2}
\end{equation}
We say that the positive operator $R_{\left(G,\nu\right)}$ is \emph{decomposable
}if there is a representation $Y=U\times J$ with (\ref{eq:Ir1})
such that, for $\nu\left(\cdot\mid i\right)$ a.e. $u\in U$, the
induced IFS, 
\begin{equation}
X\ni x\longmapsto G\left(x,u,i\right)\label{eq:Ir3}
\end{equation}
is stable (\defref{Istable}); i.e., for $u$ fixed, $\exists\sigma^{\left(u\right)}\in End\left(X\right)$
such that 
\begin{equation}
\sigma^{\left(u\right)}\left(G\left(x,\left(u,i\right)\right)\right)=x,\;\forall i\in J.\label{eq:Ir4}
\end{equation}
\end{defn}
\begin{thm}
Let $\left(X,Y,G,\nu\right)$ be given as in the statement of \thmref{Im};
then the corresponding transfer operator $R=R_{\left(G,\nu\right)}$
is decomposable. 
\end{thm}
\begin{proof}
This may be proved with the use of a Zorn lemma argument; see e.g.,
\cite{MR0282379}. (Details are left to the reader.) Note that the
representation of $Y$ in (\ref{eq:Ir1})\textendash (\ref{eq:Ir2})
is not unique. 
\end{proof}
\begin{rem}
The reader will notice that the example from \propref{me} (see (\ref{eq:me3}))
is decomposable; see also \exaref{Iu}.
\end{rem}
\begin{rem}
Return to the general case, let $R=R_{\left(G,\nu\right)}$ be given
in its decomposable form with the measure $\nu$ represented as in
(\ref{eq:Ir1}) for a fixed system of weights $\left(p_{i}\right)_{i\in J}$,
$\sum_{i}p_{i}=1$. Let $\left(\pi_{n}\right)_{n\in\mathbb{Z}_{+}\cup\left\{ 0\right\} }$be
the corresponding Markov process on $\Omega_{X}=\prod_{0}^{\infty}X$;
see \thmref{sm1}. We then have the following formula for the Markov-move
$\pi_{0}\rightarrow\pi_{1}$; and similarly for $\pi_{n}\rightarrow\pi_{n+1}$: 

Let $x\in X$, and $A\in\mathscr{B}_{X}$, then 
\begin{equation}
\mathbb{P}\left(\pi_{1}\in A\mid\pi_{0}=x\right)=\sum_{i\in J}p_{i}\int_{U}\nu\left(\left\{ G\left(x,y\right)\in A\mid\pi_{U}\in du,\:\pi_{J}=i\right\} \right).\label{eq:Ir5}
\end{equation}

The Markov move is as follows: Step 1 selects $i$ with probability
$p_{i}$, and the second step selects $\pi_{1}\in A$ from $\nu\left(\cdot\mid i\right)$;
see \figref{mm}.
\end{rem}
\begin{figure}
\includegraphics[width=0.4\textwidth]{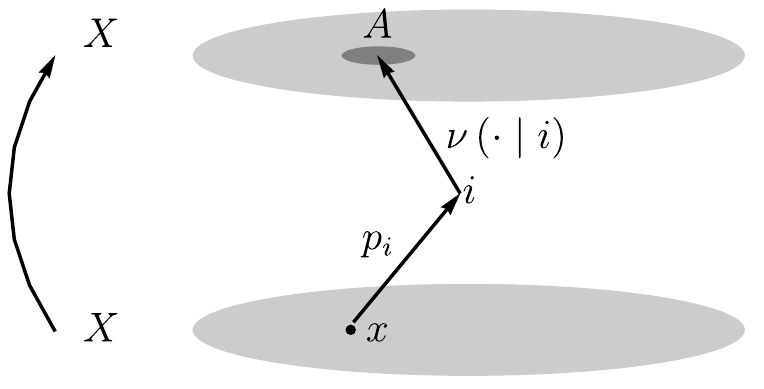}

\caption{\label{fig:mm}The Markov-move $\pi_{0}\rightarrow\pi_{1}$, see (\ref{eq:Ir5}).}

\end{figure}

\section{\label{sec:gm}Generalized multiresolutions associated to measure
spaces with endomorphism}

\subsection{\label{subsec:gmmr}Multiresolutions}

In this section we introduce the aforementioned multiresolutions,
with the scale of resolution subspaces referring to the Hilbert spaces
$L^{2}\left(\Omega,\mathbb{P}\right)$ of square integrable random
variables.

In classical wavelet theory, the accepted use is instead the Hilbert
space $L^{2}\left(\mathbb{R}\right)$, and systems of functions $\varphi$,
$\left(\psi_{i}\right)$ in $L^{2}\left(\mathbb{R}\right)$ such that
\begin{align}
\varphi\left(x\right) & =\sqrt{N}\sum_{k\in\mathbb{Z}}a_{k}\varphi\left(Nx-k\right),\;\text{and}\label{eq:wa1}\\
\psi_{i}\left(x\right) & =\sqrt{N}\sum_{k\in\mathbb{Z}}b_{k}^{\left(i\right)}\varphi\left(Nx-k\right)\label{eq:wa2}
\end{align}
where the coefficients $\left(a_{k}\right)$ and $\left(b_{k}^{\left(i\right)}\right)$
are called wavelet \emph{masking coefficients}. From this one creates
wavelet multiresolutions as follows: 
\begin{enumerate}
\item $\mathscr{H}_{n}$, $\mathscr{H}_{n}\subset\mathscr{H}_{n+1}$, $\land_{n}\mathscr{H}_{n}=\left\{ 0\right\} $,
$\lor_{n}\mathscr{H}_{n}=L^{2}\left(\mathbb{R}\right)$;
\item $\mathscr{H}_{0}=\lor span\left\{ \varphi\left(\cdot-k\right)\mid k\in\mathbb{Z}\right\} $;
\item $\exists N\in\mathbb{N}$, $N>1$, such that $\mathscr{H}_{0}\ominus\mathscr{H}_{1}=\lor\left\{ \psi_{j}\left(\cdot-k\right)\mid1\leq j<N,k\in\mathbb{Z}\right\} $;
\item $U^{k}\mathscr{H}_{0}=\mathscr{H}_{-k}$, $k\in\mathbb{Z}$, where
\begin{equation}
\left(Uf\right)\left(x\right)=\frac{1}{\sqrt{N}}f\left(\frac{x}{N}\right).\label{eq:gms1}
\end{equation}
\end{enumerate}
So if $N>1$ is fixed, the goal is the construction of functions $\psi_{1},\psi_{2},\cdots,\psi_{N-1}$
such that the corresponding triple-indexed family 
\begin{align}
\psi_{j,k,n}\left(x\right)= & N^{\frac{n}{2}}\psi_{j}\left(N^{n}x-k\right),\nonumber \\
 & j=1,\cdots,N-1,\:k,n\in\mathbb{Z},\label{eq:gms2}
\end{align}
forms a suitable frame in $L^{2}\left(\mathbb{R}\right)$; or even
an ONB. 

For more details, see \cite{MR1162107,MR1913212,MR3474523,MR3474520,JT15a1,MR3275999}.
\begin{defn}
\label{def:gmce}Let $\left(\Omega,\mathscr{F},\mathbb{P}\right)$
be a probability space, and let $\mathscr{A}\subset\mathscr{F}$ be
a sub-sigma algebra. For every $\xi\in L^{2}\left(\Omega,\mathscr{F},\mathbb{P}\right)$
we define the \emph{conditional expectation} $\mathbb{E}\left(\xi\mid\mathscr{A}\right)$
as the Radon-Nikodym derivative 
\begin{equation}
\mathbb{E}\left(\xi\mid\mathscr{A}\right):=\frac{d\left(\xi d\mathbb{P}\right)}{d\mathbb{P}\big|_{\mathscr{A}}}.\label{eq:gmc1}
\end{equation}

Note further that $\mathbb{E}\left(\cdot\mid\mathscr{A}\right)$ is
the orthogonal projection of $L^{2}\left(\Omega,\mathscr{F},\mathbb{P}\right)$
onto the closed subspace $L^{2}\left(\Omega,\mathscr{A},\mathbb{P}\big|_{\mathscr{A}}\right)$;
i.e., we have 
\[
\int_{\Omega}\varphi\,\xi\,d\mathbb{P}=\int_{\Omega}\varphi\,\mathbb{E}\left(\xi\mid\mathscr{A}\right)d\mathbb{P}
\]
for all $\varphi$ $\mathscr{A}$-measurable, and all $\xi\in L^{2}\left(\Omega,\mathscr{F},\mathbb{P}\right)$. 
\end{defn}
In our applications below we shall consider multiresolutions $\mathscr{H}_{n}\subset L^{2}\left(\Omega,\mathscr{F},\mathbb{P}\right)$
which result from filtrations $\mathscr{F}_{n}\subset\mathscr{F}$
s.t. $\mathscr{F}_{n}\subset\mathscr{F}_{n+1}$, $\bigwedge_{n}\mathscr{F}_{n}=\left\{ \emptyset,X\right\} $
mod sets of $\mathbb{P}$-measure zero; and $\bigvee_{n}\mathscr{F}_{n}=\mathscr{F}$.
For every filtration, we shall consider the corresponding conditional
expectations $\mathbb{E}\left(\cdot\mid\mathscr{F}_{n}\right):=\mathbb{E}_{n}\left(\cdot\right)$. 

\subsection{\label{subsec:gmwr}Wavelet resolutions (review)}

We shall be interested in multiresolutions, both for the standard
$L^{2}\left(\mathbb{R}^{d}\right)$ Hilbert spaces, and for the $L^{2}$
Hilbert spaces formed from those probability spaces $\left(\Omega,\mathscr{F},\mathbb{P}\right)$
we discussed in  \secref{sm}. To help draw parallels we begin with
$L^{2}\left(\mathbb{R}^{d}\right)$. In both cases, the construction
takes as starting point certain Ruelle transfer operators.

In its simplest form, a wavelet is a function $\psi$ on the real
line $\mathbb{R}$ such that the doubly indexed family $\left\{ 2^{n/2}\psi\left(2^{n}x-k\right)\right\} _{n,k\in\mathbb{Z}}$
provides a basis or frame for all the functions in a suitable space
such as $L^{2}\left(\mathbb{R}\right)$. (Below, we specialize to
the case $N=2$ for simplicity, see (\ref{eq:gms1})-(\ref{eq:gms2}).)
Since $L^{2}\left(\mathbb{R}\right)$ comes with a norm and inner
product, it is natural to ask that the basis functions be normalized
and mutually orthogonal (but many useful wavelets are not orthogonal).
The analog-to-digital problem from signal processing (see e.g., \cite{MR3474523,MR3447989})
concerns the correspondence
\begin{equation}
f\left(x\right)\longleftrightarrow c_{n,k}\label{eq:gmw1}
\end{equation}
for the wavelet representation
\begin{equation}
f\left(x\right)=\sum_{n\in\mathbb{Z}}\sum_{k\in\mathbb{Z}}c_{n,k}2^{n/2}\psi\left(2^{n}x-k\right).\label{gmw2}
\end{equation}
We will be working primarily with the Hilbert space $L^{2}\left(\mathbb{R}\right)$,
and we allow complex-valued functions. Hence the inner product $\left\langle f,g\right\rangle =\int\overline{f\left(x\right)}g\left(x\right)dx$
has a complex conjugate on the first factor in the product under the
integral sign. If $f$ represents a signal in analog form, the wavelet
coefficients $c_{n,k}$ offer a digital representation of the signal,
and the correspondence between the two sides in (\ref{eq:gmw1}) is
a new form of the analysis/synthesis problem, quite analogous to Fourier's
analysis/synthesis problem of classical mathematics (see e.g., \cite{MR2220205,MR3345364,MR3200922}).
One reason for the success of wavelets is the fact that the algorithms
for the problem (\ref{eq:gmw1}) are faster than the classical ones
in the context of Fourier.

Nonetheless, classical wavelet multiresolutions have the following
limitation: Unless the wavelet filter (in the form of a multi-band
matrix valued frequency function) under consideration satisfies some
strong restriction, the Hilbert space $L^{2}\left(\mathbb{R}^{d}\right)$
is not a receptacle for realization. In other words, the resolution
subspaces sketched in \figref{gmw1} cannot be realized as subspaces
in the standard $L^{2}\left(\mathbb{R}^{d}\right)$-space; rather
we must resort to a probability space built on a solenoid. The latter
is related to $\mathbb{R}^{d}$, but different: As we outline in the
remaining of our paper, it may be built from the same scaling which
is used in the classical case (see (\ref{eq:gmw6}) for the special
case of $d=1$), only, in the more general setting, we must instead
use a \textquotedblleft bigger\textquotedblright{} Hilbert space;
see \thmref{gmH} below for details. Using ideas from \cite{MR2097020}
it is possible to show that $\mathbb{R}^{d}$ will be embedded inside
the corresponding solenoid; see also \cite{MR1887500,MR2240643,MR3275999,MR2097020,MR3074509,MR2268116,MR2123549,MR2122446}.
For related results, see \cite{MR3404559,MR3204025,MR2888226}.

The wavelet algorithms can be cast geometrically in terms of subspaces
in Hilbert space which describe a scale of resolutions of some signal
or some picture. They are tailor-made for an algorithmic approach
that is based upon unitary matrices or upon functions with values
in the unitary matrices. Wavelet analysis takes place in some Hilbert
space $\mathscr{H}$ of functions on $\mathbb{R}^{d}$, for example,
$\mathscr{H}=L^{2}\left(\mathbb{R}^{d}\right)$. An indexed family
of closed subspaces $\left\{ \mathcal{V}_{n}\right\} _{-\infty<n<\infty}$
such that
\begin{align}
\mathcal{V}_{n} & \subset\mathcal{V}_{n+1},\;U\mathcal{V}_{n+1}\subset\mathcal{V}_{n},\;\bigcap_{n\in\mathbb{Z}}\mathcal{V}_{n}=\left\{ 0\right\} ,\;\text{and}\nonumber \\
 & \bigvee_{n\in\mathbb{Z}}\mathcal{V}_{n}=L^{2}\left(\mathbb{R}^{d}\right),\;\text{see Fig \ref{fig:gmw1} and \ref{fig:gmw2},}\label{eq:gmw3}
\end{align}
is said to offer a resolution. (To stress the \emph{variety} of spaces
in this telescoping family, we often use the word \emph{multiresolution}.)
Here the symbol $\bigvee$ denotes the closed linear span. In pictures,
the configuration of subspaces looks like \figref{gmw1}.

\begin{figure}[H]
\includegraphics[width=0.8\textwidth]{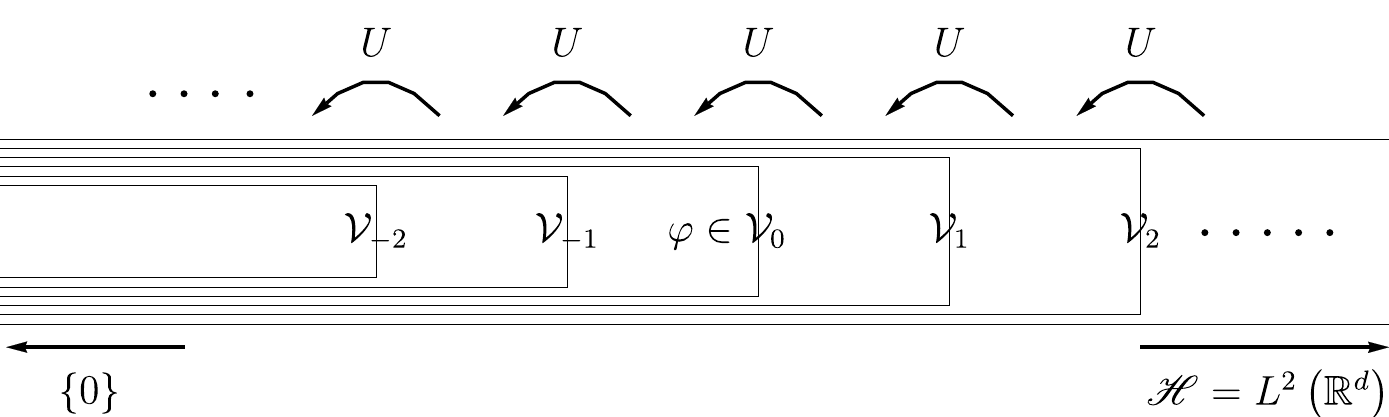}

\caption{\label{fig:gmw1}The subspaces of a resolution.}

\end{figure}

When shopping for a digital camera: just as important as the resolutions
themselves (as given here by the scale of closed subspaces $\mathcal{V}_{n}$)
are the associated spaces of \emph{detail}. (See \figref{2d} below.)
As expected, the details of a signal represent the relative complements
between the two resolutions, a coarser one and a more \emph{refined}
one. 

\textbf{Starting with the Hilbert-space approach to signals}, we are
led to the following closed subspaces (relative orthogonal complements):
\begin{align}
\mathcal{W}_{n}:= & \mathcal{V}_{n+1}\ominus\mathcal{V}_{n}\label{eq:gmw4}\\
= & \left\{ f\in\mathcal{V}_{n}\::\:\left\langle f,h\right\rangle =0,\:h\in\mathcal{V}_{n}\right\} ,\nonumber 
\end{align}
and the signals in these intermediate spaces $\mathcal{W}_{n}$ then
constitute the amount of detail which must be added to the resolution
$\mathcal{V}_{n}$ in order to arrive at the next refinement $\mathcal{V}_{n+1}$.
In \figref{gmw2}, the intermediate spaces $\mathcal{W}_{n}$ of (\ref{eq:gmw4})
represent incremental details in the resolution. See also \cite{MR2362796,MR3074509,MR3288030,MR3103223}.

\begin{figure}[H]
\includegraphics[width=0.8\textwidth]{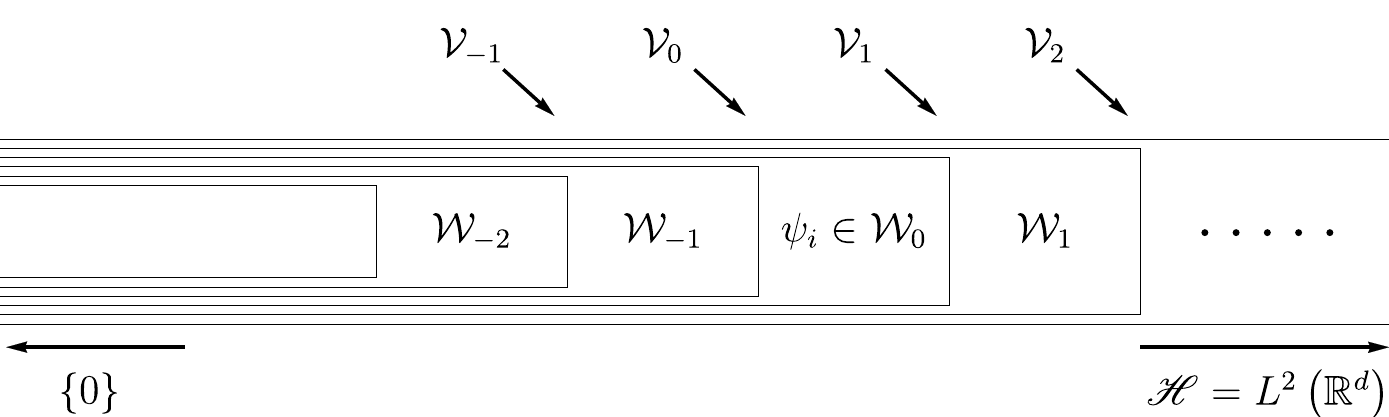}

\caption{\label{fig:gmw2}Incremental Detail.}
\end{figure}

The simplest instance of this is the one which Haar discovered in
1910 \cite{Haar1910} for $L^{2}\left(\mathbb{R}\right)$. There,
for each $n\in\mathbb{Z}$, $\mathcal{V}_{n}$ represents the space
of all step functions with step size $2^{-n}$, i.e., the functions
$f$ on $\mathbb{R}$ which are constant in each of the dyadic intervals
$j2^{-n}\leq x<\left(j+1\right)2^{-n}$, $j=0,\ldots,2^{n}-1$, and
their integral translates, and which satisfy $\left\Vert f\right\Vert ^{2}=\int_{-\infty}^{\infty}\left|f\left(x\right)\right|^{2}dx<\infty$. 

\begin{figure}[H]
\noindent \begin{centering}
\includegraphics[width=0.3\textwidth]{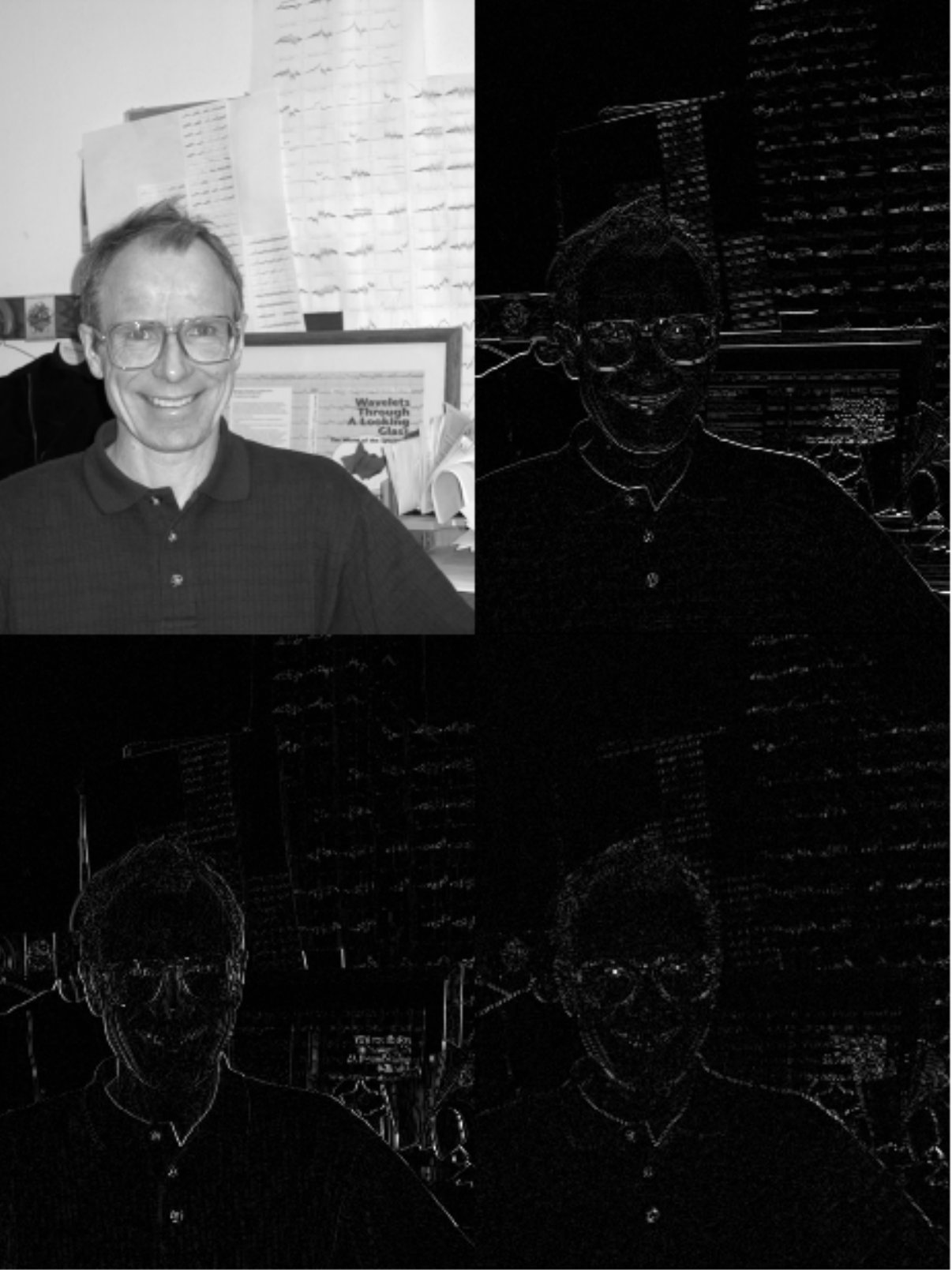}
\par\end{centering}
\caption{\label{fig:2d}A coarser resolution in three directions in the plane,
filtering in directions, $x,y$, and diagonal; \textemdash{} corresponding
dyadic scaling in each coordinate direction. (Image cited from M.-S.
Song, ``\emph{Wavelet Image Compression}'' in \cite{JLH06}.) }
\end{figure}

An operator $U$ in a Hilbert space is \emph{unitary} if it is onto
and preserves the norm or, equivalently, the inner product. Unitary
operators are invertible, and $U^{-1}=U^{*}$ where the $*$ refers
to the adjoint. Similarly, the orthogonality property for a projection
$P$ in a Hilbert space may be stated purely algebraically as $P=P^{2}=P^{*}$.
The adjoint $*$ is also familiar from matrix theory, where $\left(A^{*}\right)_{i,j}=\overline{A_{j,i}}$:
in words, the $*$ refers to the operation of transposing and taking
the complex conjugate. In the matrix case, the norm on $\mathbb{C}^{n}$
is $(\sum_{k}\left|x_{k}\right|^{2})^{1/2}$. In infinite dimensions,
there are isometries which map the Hilbert space into a proper subspace
of itself. 

For Haar's case we can scale between the resolutions using $f\left(x\right)\mapsto f\left(x/2\right)$,
which represents a dyadic scaling. 

To make it unitary, take
\begin{equation}
U=U_{2}:f\longmapsto2^{-\frac{1}{2}}f\left(\frac{x}{2}\right),\label{eq:gmw6}
\end{equation}
which maps each space $\mathcal{V}_{n}$ onto the next coarser subspace
$\mathcal{V}_{n-1}$, and $\left\Vert Uf\right\Vert =\left\Vert f\right\Vert $,
$f\in L^{2}\left(\mathbb{R}\right)$. This can be stated geometrically,
using the respective \emph{orthogonal projections} $P_{n}$ onto the
resolution spaces $\mathcal{V}_{n}$, as the identity 
\begin{equation}
UP_{n}U^{-1}=P_{n-1}.\label{eq:gmw7}
\end{equation}
And (\ref{eq:gmw7}) is a basic geometric reflection of a self-similarity
feature of the cascades of wavelet approximations (see e.g., \cite{MR1913212,MR1162107,MR1711343,MR2097020,MR3484392}).
It is made intuitively clear in Haar's simple but illuminating example.
The important fact is that this geometric self-similarity, in the
form of (\ref{eq:gmw7}), holds completely generally. See Sections
\ref{subsec:mul}, \ref{sec:us} and \ref{sec:me} below.

\subsection{\label{subsec:mul}Multiresolutions in $L^{2}\left(\Omega,\mathscr{C},\mathbb{P}\right)$}

Here we aim to realize multiresolutions in probability spaces $\left(\Omega,\mathscr{F},\mathbb{P}\right)$;
and we now proceed to outline the details.

We first need some preliminary facts and lemmas. 
\begin{lem}
\label{lem:gw1}Let $\left(\Omega,\mathscr{F},\mathbb{P}\right)$
be a probability space, and let $A:\Omega\rightarrow X$ be a random
variable with values in a fixed measure space $\left(X,\mathscr{B}_{X}\right)$,
then $V_{A}f:=f\circ A$ defines an isometry $L^{2}\left(X,\mu_{A}\right)\rightarrow L^{2}\left(\Omega,\mathbb{P}\right)$
where $\mu_{A}$ is the law (distribution) of $A$, i.e., $\mu_{A}\left(\Delta\right):=\mathbb{P}\left(A^{-1}\left(\Delta\right)\right)$,
$\forall\Delta\in\mathscr{B}_{X}$; and $V_{A}^{*}\left(\psi\right)\left(x\right)=\mathbb{E}_{\left\{ A=x\right\} }\left(\psi\mid\mathscr{F}_{A}\right)$,
for all $\psi\in L^{2}\left(\Omega,\mathbb{P}\right)$, and all $x\in X$. 
\end{lem}

We shall apply \lemref{gw1} to the case when $\left(\Omega,\mathscr{F},\mathbb{P}\right)$
is realized on an infinite product space as follows:
\begin{defn}
\label{def:gm1}Let $\left(\Omega_{X},\mathscr{F},\mathbb{P}\right)$
be a probability space, where $\Omega_{X}=\prod_{n=0}^{\infty}X$.
Let $\pi_{n}:\Omega_{X}\rightarrow X$ be the random variables given
by 
\begin{equation}
\pi_{n}\left(x_{0},x_{1},x_{2},\cdots\right)=x_{n},\;\forall n\in\mathbb{N}_{0}.\label{eq:gm1}
\end{equation}
The sigma-algebra generated by $\pi_{n}$ will be denoted $\mathscr{F}_{n}$,
and the isometry corresponding to $\pi_{n}$ will be denoted $V_{n}$. 
\end{defn}
\begin{rem}
Suppose the measure space $\left(X,\mathscr{B}_{X}\right)$ in \lemref{gw1}
is specialized to $\left(\mathbb{R},\mathscr{B}\right)$; it is then
natural to consider Gaussian probability spaces $\left(\Omega,\mathscr{F},\mathbb{P}\right)$
where $\Omega$ is a suitable choice of sample space, and $A:\Omega\rightarrow X$
is replaced with Brownian motion $B_{t}:\Omega\rightarrow\mathbb{R}$,
see \cite{MR562914,MR1075229,MR3424704,MR3402824}. We instead consider
samples 
\[
0<t_{1}<t_{2}<\cdots<t_{n},
\]
and functions $F$ on $\mathbb{R}^{n}$ with now $f\rightarrow f\circ A$
replaced with a suitable Malliavin derivative
\begin{equation}
DF_{n}\left(B_{\varphi_{1}},\cdots,B_{\varphi_{n}}\right)=\sum_{i=1}^{n}\frac{\partial F_{n}}{\partial x_{i}}\left(B_{\varphi_{1}},\cdots,B_{\varphi_{n}}\right)\varphi_{i},\label{eq:gmma}
\end{equation}
where $B_{\varphi}=\int\varphi\left(t\right)dB_{t}$. 

We computed the adjoint of (\ref{eq:gmma}) in \cite{axioms5020012}
and identified it as a multiple Ito-integral. For more details, we
refer the reader to the papers \cite{MR3212681,MR3200725,MR775042,MR1811249,MR3158567},
and also see \cite{MR1642391,hida2013white}.
\end{rem}
\begin{defn}
\label{def:gmR}Let $R$ be a positive transfer operator, i.e., $f\geq0\Rightarrow Rf\geq0$,
$R\mathbbm{1}=\mathbbm{1}$ (see \secref{ms}), let $\lambda$ be
a probability measure on a fixed measure space $\left(X,\mathscr{B}_{X}\right)$.
We further assume that
\begin{equation}
R\left(\left(f\circ\sigma\right)g\right)=fR\left(g\right),\;\forall f,g\in\mathscr{F}\left(X,\mathscr{B}_{X}\right).\label{eq:gmas1}
\end{equation}
Denote $\mu\left(\cdot\mid x\right)$, $x\in X$, the conditional
measures determined by 
\begin{equation}
Rf\left(x\right)=\int_{X}f\left(y\right)\mu\left(dy\mid x\right),\label{eq:Rx}
\end{equation}
for all $f\in C\left(X\right)$, representing $R$ as an integral
operator. Set 
\begin{align}
\mu\left(B\mid x\right):= & R\left(\chi_{B}\right)\left(x\right),\;\forall B\in\mathscr{B}_{X}\nonumber \\
= & \mathbb{P}\left(\pi_{1}\in B\mid\pi_{0}=x\right).\label{eq:Rx1}
\end{align}

Note the RHS of (\ref{eq:Rx}) extends to all measurable functions
on $X$, and we shall write $R$ also for this extension. 
\end{defn}
\begin{lem}
Let $\left\{ \mu\left(\cdot\mid x\right)\right\} _{x\in X}$ be as
in (\ref{eq:Rx}), and $W:=\frac{d\lambda R}{d\lambda}=$ Radon-Nikodym
derivative. If $B\in\mathscr{B}_{X}$ then 
\[
\int_{X}\mu\left(B\mid x\right)d\lambda\left(x\right)=\int_{B}W\left(x\right)d\lambda\left(x\right).
\]
\end{lem}
\begin{proof}
Let $B\in\mathscr{B}_{X}$, then 
\begin{align*}
\text{LHS} & =\int_{X}R\left(\chi_{B}\right)\left(x\right)d\lambda\left(x\right)\\
 & =\int_{X}\chi_{B}d\left(\lambda R\right)=\int_{B}W\left(x\right)d\lambda\left(x\right)=\text{RHS}.
\end{align*}
\end{proof}
\begin{lem}
Suppose $R$ has a representation 
\[
R\left(\chi_{B}\right)\left(x\right)=\mu\left(B\mid x\right),\;B\in\mathscr{B}_{X},\;x\in X.
\]
Then the following are equivalent: 
\begin{enumerate}
\item $R\left[\left(f\circ\sigma\right)g\right]\left(x\right)=f\left(x\right)R\left(g\right)\left(x\right)$,
$\forall x\in X$, $\forall f,g\in\mathscr{F}\left(X,\mathscr{B}\right)$; 
\item $\mu\left(\sigma^{-1}\left(A\right)\cap B\mid x\right)=\chi_{A}\left(x\right)\mu\left(B\mid x\right)$,
$\forall A,B\in\mathscr{B}$, $\forall x\in X$. 
\end{enumerate}
\end{lem}
\begin{proof}
Recall that, by assumption, $\left(Rf\right)\left(x\right)=\int_{X}f\left(x\right)\mu\left(dy\mid x\right)$.
The conclusion follows by setting $f=\chi_{A}$, and $g=\chi_{B}$. 
\end{proof}
\begin{prop}
Let $\left\{ \mu\left(\cdot\mid x\right)\right\} _{x\in X}$ be the
Markov process indexed by $x\in X$ (see (\ref{eq:Rx})), where $\left(X,\mathscr{B}_{X}\right)$
is a fixed measure space, and let $\mathbb{P}$ be the corresponding
path space measure (see, e.g., \cite{MR832433,hida2013white}) determined
by (\ref{eq:pmm1})-(\ref{eq:pmm2}). Let $\sigma\in End\left(X,\mathscr{B}_{X}\right)$
as in Def. \ref{def:end}. Then 
\begin{align}
\text{suppt}\left(\mathbb{P}\right) & \subset Sol_{\sigma}\left(X\right)\nonumber \\
 & \Updownarrow\\
\mathbb{P}\left(\pi_{k+1}\in B\cap\sigma^{-1}\left(A\right)\mid\pi_{k}=x\right) & =\chi_{A}\left(x\right)\mathbb{P}\left(\pi_{k+1}\in B\mid\pi_{k}=x\right).\nonumber 
\end{align}
\end{prop}
The next result will serve as a tool in our subsequent study of multiresolutions,
orthogonality relations, and scale-similarity, each induced by a given
endomorphism; the theme to be studied in detail in  \secref{me} below.
\begin{thm}
\label{thm:shift}Let $\left(X,\sigma,R,h,\lambda,W\right)$ be as
above, $W=\frac{d\lambda R}{d\lambda}$; then
\begin{enumerate}
\item $\exists!$ path space measure $\mathbb{P}$ on $Sol_{\sigma}\left(X\right)$,
such that 
\begin{equation}
L^{2}\left(X,\mu_{n}\right)\xrightarrow{\;V_{n}\;}L^{2}\left(Sol_{\sigma},\mathbb{P}\right),\;V_{n}f=f\circ\pi_{n}\label{eq:gm12}
\end{equation}
is isometric, where $\mu_{n}:=\text{dist}\left(\pi_{n}\right)$, and
$\int_{X}f\,d\mu_{n}=\int_{X}R^{n}\left(fh\right)d\lambda$; 
\item $\mathbb{P}$ has the property:
\begin{equation}
\frac{d\mathbb{P}\circ\tilde{\sigma}}{d\mathbb{P}}=W\circ\pi_{0},
\end{equation}
where $\tilde{\sigma}$ is as in (\ref{eq:sig2}). 
\end{enumerate}
\end{thm}
\begin{proof}
See \cite{JT15a1,MR3275999}.
\end{proof}
\begin{lem}
\label{lem:gmproj}Let $\Omega_{X}$, $\mathscr{F}$, $\mathbb{P}$,
$R$, $h$, $\lambda$ be as above, assume $R\mathbbm{1}=\mathbbm{1}$.
Let $V_{n}:L^{2}\left(X,\mu_{n}\right)\rightarrow L^{2}\left(Sol_{\sigma},\mathbb{P}\right)$
be the isometry in (\ref{eq:gm12}) (also see \defref{gm1}). Then
$V_{n}V_{n}^{*}$ is a projection in $L^{2}\left(Sol_{\sigma},\mathbb{P}\right)$,
and it is the conditional expectation on $\mathscr{H}_{n}$, i.e.,
\begin{equation}
V_{n}V_{n}^{*}\psi=\mathbb{E}\left(\psi\mid\mathscr{F}_{n}\right),\;\forall\psi\in L^{2}\left(Sol_{\sigma},\mathbb{P}\right).
\end{equation}
Moreover,
\begin{equation}
\mathbb{E}\left(\psi\mid\mathscr{F}_{n}\right)=\left(V_{n}^{*}\psi\right)\circ\pi_{n}\xrightarrow[\;n\rightarrow\infty\;]{}\psi,
\end{equation}
i.e., 
\[
\left\Vert \psi-\left(V_{n}^{*}\psi\right)\circ\pi_{n}\right\Vert _{L^{2}\left(\mathbb{P}\right)}\xrightarrow[\;n\rightarrow\infty\;]{\text{in norm}}0.
\]
\end{lem}
In order to get an orthogonal decomposition relative to the detail
spaces 
\begin{equation}
\mathscr{D}_{n}=\mathscr{H}_{n}\ominus\mathscr{H}_{n-1}=\left\{ \psi\in\mathscr{H}_{n}\mid\psi\perp\mathscr{H}_{n-1}\right\} ,
\end{equation}
we shall use that 
\begin{align}
\mathbb{E}\left(\cdot\mid\mathscr{F}_{n}\right)= & \text{the orthogonal projection in }L^{2}\left(Sol_{\sigma},\mathbb{P}\right)\\
 & \text{onto }\mathscr{H}_{n},\nonumber 
\end{align}
and so the orthogonal projection onto $\mathscr{D}_{n}$ is 
\begin{equation}
\mathbb{E}\left(\cdot\mid\mathscr{F}_{n}\right)-\mathbb{E}\left(\cdot\mid\mathscr{F}_{n-1}\right).\label{eq:gmp1}
\end{equation}

\begin{lem}
\label{lem:gmproj1}Assume $R\mathbbm{1}=\mathbbm{1}$. For all $f\in\mathscr{F}\left(X,\mathscr{B}_{X}\right)$,
we have
\begin{equation}
VV^{*}\left(f\circ\pi_{n+k}\right)=\left[R^{k}\left(f\right)-R^{k+1}\left(f\right)\circ\sigma\right]\circ\pi_{n}.
\end{equation}
\end{lem}
\begin{proof}
Note that, for all $f,g\in\mathscr{F}\left(X,\mathscr{B}_{X}\right)$,
\begin{align*}
 & \int_{Sol_{\sigma}}\left(g\circ\pi_{n}\right)\left(f\circ\pi_{n+k}\right)d\mathbb{P}\\
= & \int_{Sol_{\sigma}}\left(\left(g\circ\sigma^{k}\right)f\right)\circ\pi_{n+k}d\mathbb{P}=\int_{X}R^{n+k}\left(\left(g\circ\sigma^{k}\right)f\right)h\,d\lambda\\
= & \int_{X}R^{n}\left(gR^{k}\left(f\right)\right)h\,d\lambda=\int_{Sol_{\sigma}}\left(g\circ\pi_{n}\right)\left(R^{k}\left(f\right)\circ\pi_{n}\right)d\mathbb{P},
\end{align*}
and so $\mathbb{E}\left(f\circ\pi_{n+k}\mid\mathscr{F}_{n}\right)=R^{k}\left(f\right)\circ\pi_{n}$. 

Apply (\ref{eq:gmp1}) to $f\circ\pi_{n+k}$, then 
\begin{align*}
 & \mathbb{E}\left(f\circ\pi_{n+k}\mid\mathscr{F}_{n}\right)-\mathbb{E}\left(f\circ\pi_{n+k}\mid\mathscr{F}_{n-1}\right)\\
= & R^{k}\left(f\right)\circ\pi_{n}-R^{k+1}\left(f\right)\circ\pi_{n-1}\\
= & \left[R^{k}\left(f\right)-R^{k+1}\left(f\right)\circ\sigma\right]\circ\pi_{n},
\end{align*}
which is assertion. 
\end{proof}

\begin{lem}
Assume $R\mathbbm{1}=\mathbbm{1}$, then 
\[
R\left[f-R\left(f\right)\circ\sigma\right]\equiv0,\;\forall f\in\mathscr{F}\left(X,\mathscr{B}_{X}\right).
\]
\end{lem}
\begin{proof}
It follows from (\ref{eq:gmas1}) that 
\[
R\left(R\left(f\right)\circ\sigma\right)=R\left(R\left(f\right)\circ\sigma\mathbbm{1}\right)=R\left(f\right)R\left(\mathbbm{1}\right)=R\left(f\right).
\]
\end{proof}
\begin{rem}
\label{rem:gmR}The path space measure from (\ref{eq:pmm1}) (see,
e.g., \cite{MR832433,hida2013white}) can be formulated as follows: 

Assume $R'\mathbbm{1}=\mathbbm{1}$, and $\int_{X}h\,d\lambda=1$,
and let $\mathbb{P}$ be determined by 
\begin{align}
 & \int_{\Omega_{X}}\left(f_{0}\circ\pi_{0}\right)\left(f_{1}\circ\pi_{1}\right)\cdots\left(f_{n}\circ\pi_{n}\right)d\mathbb{P}\nonumber \\
= & \int_{X}f_{0}\left(x\right)R'\left(f_{1}R\left(f_{2}\cdots R'\left(f_{n}\right)\right)\cdots\right)\left(x\right)h\left(x\right)d\lambda\left(x\right).\label{eq:pmm3}
\end{align}

The two constructions in (\ref{eq:pmm1}) and (\ref{eq:pmm3}) are
equivalent and generate the \emph{same} path space measure. See \thmref{gmR}
below. 
\end{rem}

\subsection{Renormalization}

The purpose of the next result is to show that in the study of path-space
measures associated to positive transfer operators $R$ one may in
fact reduce to the case when $R$ is assumed normalized; see (\ref{eq:gmR1})
in the statement of the theorem. The result will be used in the remaining
of our paper.
\begin{thm}
\label{thm:gmR}Let $\left(X,\mathscr{B}_{X},R,h,\lambda\right)$
be as above, i.e., $Rh=h$, $h\geq0$, $\int_{X}h\,d\lambda=1$, and
let $\mathbb{P}$ be the corresponding probability measure on $\Omega_{X}=\prod_{n=0}^{\infty}\left(X,\mathscr{B}_{X}\right)$
equipped with its cylinder sigma-algebra $\mathscr{C}$. 

Define $R'$ as follows:
\begin{equation}
R'\left(f\right):=\frac{R\left(fh\right)}{h},\;\forall f\in\mathscr{F}\left(X,\mathscr{B}_{X}\right),\label{eq:gmR1}
\end{equation}
then $R'$ is well defined, $R'\left(\mathbbm{1}\right)=\mathbbm{1}$,
and $\left(R',\lambda\right)$ generates the \uline{same} probability
space $\left(\Omega_{X},\mathscr{C},\mathbb{P}\right)$. (See also
\remref{gmR}.)
\end{thm}
\begin{proof}
To see that $R'$ (in (\ref{eq:gmR1})) is well defined, note that
a repeated application of Schwarz yields: 
\[
\left|R\left(fh\right)\right|\leq\left(R\left(f^{2}h\right)\right)^{\frac{1}{2}}h^{\frac{1}{2}}\leq\cdots\leq R\left(f^{2^{n}}h\right)^{\frac{1}{2^{n}}}h^{\frac{1}{2}+\cdots+\frac{1}{2^{n}}}
\]
for all $f\in\mathscr{F}\left(X,\mathscr{B}_{X}\right)$, and all
$n\in\mathbb{N}$.

For each $n\in\mathbb{Z}_{+}$, consider $f_{0},f_{1},\cdots,f_{n}$
in $\mathscr{F}\left(X,\mathscr{B}_{X}\right)$. We note that $\mathbb{P}$
from $\left(R,h,\lambda\right)$ is determined by the conditional
measures
\begin{align}
 & \int_{\Omega_{X}}\left(f_{0}\circ\pi_{0}\right)\left(f_{1}\circ\pi_{1}\right)\cdots\left(f_{n}\circ\pi_{n}\right)d\mathbb{P}\nonumber \\
= & \int_{X}f_{0}\left(x\right)R\left(f_{1}R\left(f_{2}\cdots R\left(f_{n}h\right)\cdots\right)\right)\left(x\right)d\lambda\left(x\right),\;\text{and}\label{eq:gmR2}
\end{align}
\[
\int_{\Omega_{X}}\left(f\circ\pi_{0}\right)d\mathbb{P}=\int_{X}f\,h\,d\lambda,
\]
while the measures on $\left(\Omega_{X},\mathscr{C}\right)$ determined
by $R'$ from (\ref{eq:gmR1}) are 
\begin{equation}
\int_{X}f_{0}\left(x\right)R'\left(f_{1}R'\left(f_{2}\cdots R'\left(f_{n}\right)\cdots\right)\right)\left(x\right)h\left(x\right)d\lambda\left(x\right).\label{eq:gmR3}
\end{equation}

But an induction by $n$ shows that the integrals in (\ref{eq:gmR3})
agree with the RHS in (\ref{eq:gmR2}) for all $n\in\mathbb{N}$,
and all $f_{0},f_{1},\cdots,f_{n}$ in $\mathscr{F}\left(X,\mathscr{B}_{X}\right)$.
We then conclude from Kolmogorov consistency that the two measures
on $\left(\Omega_{X},\mathscr{C}\right)$ agree; i.e., that $\left(R,h,\lambda\right)$
and $\left(R',\mathbbm{1},h\,d\lambda\right)$ induce the same path
space measure on $\left(\Omega_{X},\mathscr{C}\right)$, i.e., we
get the \emph{same} $\mathbb{P}$ for the unnormalized $R$ as from
its normalized counterpart. See, e.g., \cite{MR562914,MR3272038,MR0279844}.
\end{proof}
\begin{thm}
\label{thm:gmH}Let $\Omega_{X}$, $\mathscr{F}$, $\mathbb{P}$,
$R$, $h$, $\lambda$ be as specified above, such that $R\mathbbm{1}=\mathbbm{1}$,
and $\mathbb{P}$ is determined by (\ref{eq:pmm3}). Set 
\[
\mathscr{H}_{n}:=\bigvee\left\{ f\circ\pi_{n}\mid f\in L^{2}\left(X,\mathscr{B}_{X},\lambda\right)\right\} .
\]
Let $\sigma:X\rightarrow X$ be a measurable endomorphism mapping
$X$ \emph{onto} itself. Assume further that 
\begin{enumerate}
\item $\bigcap_{n=1}^{\infty}\sigma^{-n}\left(\mathscr{B}_{X}\right)=\left\{ \emptyset,X\right\} $
mod sets of $\lambda$-measure zero; 
\item $R\left(\left(f\circ\sigma\right)g\right)=fR\left(g\right)$, $\forall f,g\in\mathscr{F}\left(X,\mathscr{B}_{X}\right)$. 
\end{enumerate}
Then the resolution space $\mathscr{H}_{n}$ has an orthogonal decomposition
in $L^{2}\left(Sol_{\sigma},\mathbb{P}\right)$ as follows (\figref{gm1}):
Setting 
\begin{equation}
\mathscr{D}_{k}=\mathscr{H}_{k}\ominus\mathscr{H}_{k-1}\left(=\text{detail subspace}\right),\;k=1,\cdots,n;
\end{equation}
then 
\begin{equation}
f\circ\pi_{n}=\underset{\in\mathscr{D}_{n}}{\underbrace{\left(f-R\left(f\right)\circ\sigma\right)\circ\pi_{n}}}+\underset{\in\mathscr{D}_{n-1}}{\underbrace{\left(R\left(f\right)-R^{2}\left(f\right)\circ\sigma\right)\circ\pi_{n-1}}}+\cdots\label{eq:gm17}
\end{equation}
is the corresponding orthogonal decomposition for arbitrary vectors
in the $n^{th}$ resolution subspace in $L^{2}\left(Sol_{\sigma},\mathbb{P}\right)$.

\end{thm}
\begin{proof}
Note that
\begin{align*}
 & \int_{Sol_{\sigma}}\left(g\circ\pi_{n-1}\right)\left(f-R\left(f\right)\circ\sigma\right)\circ\pi_{n}d\mathbb{P}\\
= & \int_{Sol_{\sigma}}\left(g\circ\pi_{n-1}\right)\left(f\circ\pi_{n}\right)d\mathbb{P}-\int_{Sol_{\sigma}}\left(g\circ\pi_{n-1}\right)R\left(f\right)\circ\underset{\pi_{n-1}}{\underbrace{\sigma\circ\pi_{n}}}d\mathbb{P}\\
= & \int_{X}R^{n-1}\left(gR\left(f\right)\right)h\,d\lambda-\int_{X}R^{n-1}\left(gR\left(f\right)\right)h\,d\lambda=0,\;\forall f,g\in\mathscr{F}\left(X,\mathscr{B}_{X}\right),
\end{align*}
and the conclusion follows by induction. Also see \lemref{gmproj1}.
\end{proof}

\begin{figure}[H]
\includegraphics[width=0.9\textwidth]{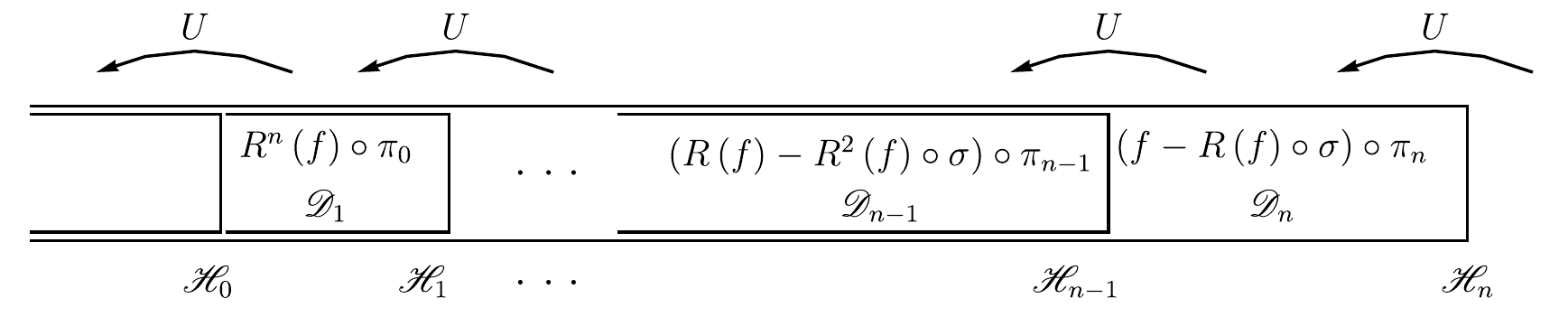}

\caption{\label{fig:gm1}}
\end{figure}

\begin{example}
For $f\in\mathscr{F}\left(X,\mathscr{B}_{X}\right)$, apply (\ref{eq:gm17})
to $f\circ\pi_{1}$ then 
\[
f\circ\pi_{1}=\left(f-R\left(f\right)\circ\sigma\right)\circ\pi_{1}+R\left(f\right)\circ\pi_{0},
\]
and by Parseval's identity, 
\[
\int_{X}R\left(f^{2}\right)hd\lambda=\int_{X}\left(R\left(f^{2}\right)-R\left(f\right)^{2}\right)h\,d\lambda+\int_{X}R\left(f\right)^{2}h\,d\lambda.
\]
\end{example}
\begin{rem}[Analogy with Brownian motion]
Let $\left(B_{t}\right)_{t\in\left[0,T\right]}$ be the standard
Brownian motion, so that $\mathbb{E}\left(B_{s}B_{t}\right)=s\wedge t=\min\left(s,t\right)$,
then 
\[
\mathbb{E}\left(\left|\int_{0}^{T}f\left(B_{t}\right)dB_{t}\right|^{2}\mid\mathscr{F}_{0}\right)=\int_{0}^{T}\left|f\left(t\right)\right|^{2}dt.
\]
 Note that in our current setting, we have
\begin{align*}
\mathbb{E}\left(f^{2}\circ\pi_{n}\mid\mathscr{F}_{0}\right) & =\int_{X}\left|f\right|^{2}d\mu_{n}\\
 & =\sum_{k=0}^{n}\int_{X}R^{k}\left(R\left(f^{2}\right)-R\left(f\right)^{2}\right)d\lambda.
\end{align*}
Also see \cite{MR562914,MR841237,MR3424704,MR3402824}. 
\end{rem}
\begin{lem}
For all $f\in L^{\infty}\left(X,\mathscr{B}_{X},\lambda\right)$,
let $\rho\left(f\right):=$ multiplication by $f\circ\pi_{0}$, as
an operator in $L^{2}\left(\Omega_{X},\mathscr{F},\mathbb{P}\right)$,
then the action of $\left\{ \rho\left(f\right)\right\} _{f\in L^{2}\left(X\right)}$
is as follows: 

Every subspace $\mathscr{H}_{n}$ is invariant under $\rho\left(f\right)$,where
\begin{align}
\rho\left(f\right)\big|_{\mathscr{H}_{n}} & =M_{f\circ\sigma^{n}}=\text{multiplication by }f\circ\sigma^{n}\label{eq:gmr1}\\
\rho\left(f\right)\big|_{\mathscr{D}_{n}} & =0,\quad\mathscr{D}_{n}:=\mathscr{H}_{n}\ominus\mathscr{H}_{n-1}.\label{eq:gmr2}
\end{align}
\end{lem}
\begin{proof}
(Sketch) Note that 
\[
\rho\left(f\right)g\circ\pi_{n}=\left(f\circ\pi_{0}\right)\left(g\circ\pi_{n}\right)=\left(\left(f\circ\sigma^{n}\right)g\right)\circ\pi_{n}.
\]
The conclusion follows from this.
\end{proof}

\section{\label{sec:us}Unitary scaling in $L^{2}\left(\Omega,\mathscr{C},\mathbb{P}\right)$ }

Let $\left(X,\mathscr{B}\right)$ be a measure space, and let $R$
be a positive operator in $\mathscr{F}\left(X,\mathscr{B}\right)$.
Let $h$ be harmonic, i.e., $h\geq0$, $Rh=h$; and let $\lambda$
be a positive measure on $\left(X,\mathscr{B}\right)$ s.t. 
\begin{equation}
\int_{X}h\left(x\right)d\lambda\left(x\right)=1.\label{eq:us1}
\end{equation}

Let $\mathbb{P}$ be the probability measure on $\left(\Omega_{X},\mathscr{C}\right)$
from sect \ref{subsec:mul}, i.e., relative to 
\begin{equation}
\pi_{n}\left(x_{0},x_{1},x_{2},\cdots\right)=x_{n},\;n\in\mathbb{Z}_{+}\cup\left\{ 0\right\} ,\label{eq:us2}
\end{equation}
$h\,d\lambda$ is the law (distribution) of $\pi_{0}$, while 
\begin{align}
 & \int_{X}f_{0}\left(x\right)R\left(f_{1}R\left(f_{2}\cdots R\left(f_{n}h\right)\cdots\right)\right)\left(x\right)d\lambda\left(x\right)\label{eq:us3}\\
= & \mathbb{E}\left(\left(f_{0}\circ\pi_{0}\right)\left(f_{1}\circ\pi_{1}\right)\cdots\left(f_{n}\circ\pi_{n}\right)\right)\nonumber 
\end{align}
for all $n\in\mathbb{Z}_{+}$, and $\left\{ f_{i}\right\} _{i=0}^{n}$
in $\mathscr{F}\left(X,\mathscr{B}\right)$. 

\needspace{3\baselineskip}
\begin{lem}
~
\begin{enumerate}
\item Let $s$ be the shift in $\Omega_{X}$, 
\begin{equation}
s\left(x_{0},x_{1},x_{2},\cdots\right):=\left(x_{1},x_{2},x_{3},\cdots\right),\label{eq:us4}
\end{equation}
then the following are equivalent:
\begin{enumerate}
\item \label{enu:us1}$\lambda R\ll\lambda$, and $\frac{d\lambda R}{d\lambda}=W$;
and
\item \label{enu:us2}$\mathbb{P}\circ s\ll\mathbb{P}$, and $\frac{d\mathbb{P}\circ s^{-1}}{d\mathbb{P}}=W\circ\pi_{0}.$
\end{enumerate}
\item If the conditions hold, then 
\begin{equation}
U_{1}\xi=\left(\xi\circ s\right)\frac{1}{\sqrt{W\circ\pi_{1}}},\label{eq:us5}
\end{equation}
for all $\xi\in L^{2}\left(\Omega_{X},\mathscr{C},\mathbb{P}\right)$,
defines a co-isometry. 
\item The operator $U_{1}$ in (\ref{eq:us5}) is unitary if 
\begin{equation}
\lambda\left(\left\{ W=0\right\} \right)=0,\label{eq:us6}
\end{equation}
and if there is an endomorphism $\sigma$ such that $s=\tilde{\sigma}^{-1}$. 
\end{enumerate}
\end{lem}
\begin{proof}
Most of the arguments are already contained in the previous sections.
Given $\left(R,h,\lambda\right)$ as stated, the corresponding measure
$\mathbb{P}$ on $\left(\Omega_{X},\mathscr{C}\right)$ is determined
by (\ref{eq:us3}) and Kolmogorov consistency \cite{MR562914,MR3272038,MR0279844}. 

And it then also follows from (\ref{eq:us3}) that the two conditions
(\ref{enu:us1})\textendash (\ref{enu:us2}) in the lemma are equivalent.
The assertion about $U_{1}$ in (\ref{eq:us5}) follows from this. 
\end{proof}
We shall be primarily interested in the case of endomorphisms, i.e.,
we assume that there is an endomorphism $\sigma$ of $X$ as in (\ref{enu:ms1})-(\ref{enu:ms2})
of \defref{ms}, with solenoid action (\defref{end}):
\begin{align*}
\tilde{\sigma}\left(x_{0},x_{1},x_{2},\cdots\right) & =\left(\sigma\left(x_{0}\right),x_{1},x_{2},\cdots\right),\;\text{and}\\
\tilde{\sigma}^{-1}\left(x_{0},x_{1},x_{2},\cdots\right) & =\left(x_{1},x_{2},x_{3},\cdots\right)=s.
\end{align*}

In that case, condition (\ref{enu:us2}) in the lemma reads as follows
\begin{equation}
\frac{d\left(\mathbb{P}\circ\tilde{\sigma}\right)}{d\mathbb{P}}=W\circ\pi_{0},\label{eq:us7}
\end{equation}
and we get the unitary operator
\begin{equation}
U\xi=\left(\xi\circ\tilde{\sigma}\right)\sqrt{W\circ\pi_{0}},\label{eq:us8}
\end{equation}
and the adjoint operator in $L^{2}\left(Sol_{\sigma}\left(X\right),\mathscr{C},\mathbb{P}\right)$
\begin{equation}
U^{*}\xi=\left(\xi\circ\tilde{\sigma}^{-1}\right)\frac{1}{\sqrt{W\circ\pi_{1}}}.\label{eq:us9}
\end{equation}
In other words, the adjoint operator $U^{*}$ in (\ref{eq:us9}) is
the restriction of $U_{1}$ from (\ref{eq:us5}). 

\begin{proof}[Proof of the assertion in connection with the formula (\ref{eq:us8})-(\ref{eq:us9}). ]
~

We must verify the following identity (\ref{eq:us9a}) for all $\xi,\eta\in L^{2}\left(Sol_{\sigma},\mathbb{P}\right)$,
where 
\begin{equation}
\int_{Sol_{\sigma}}\left(\xi\circ\tilde{\sigma}\right)\sqrt{W\circ\pi_{0}}\,\eta\,d\mathbb{P}=\int_{Sol_{\sigma}}\xi\left(\eta\circ\tilde{\sigma}^{-1}\right)\frac{1}{\sqrt{W\circ\pi_{1}}}d\mathbb{P}.\label{eq:us9a}
\end{equation}
With an application of \thmref{gmR} above, we may assume without
loss of generality that $R$ is normalized. An application of \lemref{gmproj}
further shows that formula (\ref{eq:us9a}) follows from its simplification
(\ref{eq:us9b}), i.e., we may prove the following simplified version:
\begin{align}
 & \int_{Sol_{\sigma}}\left(f\circ\pi_{n}\circ\tilde{\sigma}\right)\sqrt{W\circ\pi_{0}}\left(g\circ\pi_{n+k}\right)d\mathbb{P}\nonumber \\
= & \int_{Sol_{\sigma}}\left(f\circ\pi_{n}\right)\left(g\circ\pi_{n+k}\circ\tilde{\sigma}^{-1}\right)\frac{1}{\sqrt{W\circ\pi_{1}}}d\mathbb{P};\label{eq:us9b}
\end{align}
setting $\xi=f\circ\pi_{n}$, and $\eta=g\circ\pi_{n+k}$. 

But with the use of \thmref{sm1}, we note that (\ref{eq:us9b}) in
turn simplifies to 
\begin{align}
 & \int_{X}\sqrt{W}R^{n-1}\left(fR^{k+1}\left(g\right)\right)h\,d\lambda\nonumber \\
= & \int_{X}R\left(\frac{1}{\sqrt{W}}R^{n-1}\left(fR^{k+1}\left(g\right)\right)\right)h\,d\lambda.\label{eq:us9c}
\end{align}
We finally have $\frac{d\left(\lambda R\right)}{d\lambda}=W$, so
\[
\text{RHS}_{\left(\ref{eq:us9c}\right)}=\int_{X}\sqrt{W}R^{n-1}\left(fR^{k+1}\left(g\right)\right)h\,d\lambda=\text{LHS}_{\left(\ref{eq:us9c}\right)}
\]
which is the desired conclusion. 
\end{proof}
In the remaining of this section, we specialize to the case of endomorphisms;
and we assume $\left(R,h,\lambda,\sigma\right)$ satisfy
\begin{align}
 & R\left(\left(f\circ\sigma\right)g\right)=fR\left(g\right),\;\forall f,g\in\mathscr{F}\left(X,\mathscr{B}\right),\\
 & Rh=h,\;\text{and}\\
 & \int_{X}h\,d\lambda=1.
\end{align}

As we saw in \thmref{shift}, the solenoid is shift-invariant, and
$\mathbb{P}\left(Sol_{\sigma}\left(X\right)\right)=1$. Here we show
that the induced probability space is $\left(Sol_{\sigma}\left(X\right),\mathscr{C},\mathbb{P}\right)$. 

\needspace{3\baselineskip}
\begin{thm}
~
\begin{enumerate}
\item \label{enu:usr1}Let $\left(X,\mathscr{B},R,W,h,\lambda,\sigma\right)$
be as specified above, and let $U$ be the corresponding unitary operator
from (\ref{eq:us8}). Set 
\begin{equation}
\mathbb{E}_{n}:=V_{n}V_{n}^{*},\label{eq:us13}
\end{equation}
where $V_{n}f=f\circ\pi_{n}$, $L^{2}\left(X,\mu_{n}\right)\rightarrow L^{2}\left(X,\mu_{n}\right)$
is the associated sequence of isometries (\defref{gm1}). Then
\begin{equation}
U\mathbb{E}_{n}=\mathbb{E}_{n-1}U\mathbb{E}_{n},\;\forall n\in\mathbb{Z}_{+}.\label{eq:us14}
\end{equation}
\item \label{enu:usr2}Let $\rho$ denote the representation in $L^{2}\left(Sol_{\sigma},\mathscr{C},\mathbb{P}\right)$
by multiplication operators, where 
\begin{equation}
\rho\left(f\right)\xi=\left(f\circ\pi_{0}\right)\xi,\label{eq:us15}
\end{equation}
$\forall f\in L^{\infty}\left(X,\lambda\right)$, $\forall\xi\in L^{2}\left(Sol_{\sigma},\mathscr{C},\mathbb{P}\right)$,
then 
\begin{equation}
U\rho\left(f\right)U^{*}=\rho\left(f\circ\sigma\right),\;\forall f\in L^{\infty}\left(X,\lambda\right).\label{eq:us16}
\end{equation}
\end{enumerate}
\end{thm}
\begin{proof}
(\ref{enu:usr1}) This follows from the fact that $\mathbb{E}_{n}$
in (\ref{eq:us13}) is the conditional expectation (\defref{gmce}
\& \lemref{gmproj}) onto $\mathscr{F}_{n}:=\pi_{n}^{-1}\left(\mathscr{B}\right)$,
and for $f\in\mathscr{F}\left(X,\mathscr{B}\right)$, we have 
\begin{align*}
U\left(f\circ\pi_{n}\right) & =\left(f\circ\pi_{n}\circ\tilde{\sigma}\right)\sqrt{W\circ\pi_{0}}\\
 & =\left(f\circ\pi_{n-1}\right)\sqrt{W\circ\pi_{0}}\in\mathscr{H}_{n-1},
\end{align*}
where $\mathscr{H}_{n}:=\mathbb{E}_{n}L^{2}\left(Sol_{\sigma},\mathscr{C},\mathbb{P}\right)=L^{2}\left(Sol_{\sigma},\mathscr{F}_{n},\mathbb{P}\right)$.
We also used that $\mathscr{F}_{n}\subset\mathscr{F}_{n+1}$, and
$\mathscr{H}_{n}\hookrightarrow\mathscr{H}_{n+1}$, or equivalently,
$\mathbb{E}_{n}=\mathbb{E}_{n}\mathbb{E}_{n+1}=\mathbb{E}_{n+1}\mathbb{E}_{n}$,
$\forall n\in\mathbb{Z}_{+}$. 

Proof of (\ref{enu:usr2}). Note that (\ref{eq:us16}) is equivalent
to 
\[
U\rho\left(f\right)=\rho\left(f\circ\sigma\right)U
\]
by (\ref{eq:us8})-(\ref{eq:us9}). For $\xi\in L^{2}\left(Sol_{\sigma},\mathscr{C},\mathbb{P}\right)$,
we have 
\begin{align*}
U\rho\left(f\right)\xi & =\left(\left(\left(f\circ\pi_{0}\right)\xi\right)\circ\tilde{\sigma}\right)\sqrt{W\circ\pi_{0}}\\
 & =\left(\left(f\circ\sigma\right)\circ\pi_{0}\right)\left(\xi\circ\tilde{\sigma}\right)\sqrt{W\circ\pi_{0}}=\rho\left(f\circ\sigma\right)U\xi.
\end{align*}
\end{proof}
The aim of the next subsection is to point out how the two Hilbert
spaces $L^{2}\left(\mathbb{T}\right)$, $\mathbb{T}=\mathbb{R}/\mathbb{Z}$,
and $L^{2}\left(Sol_{N}\left(\mathbb{T}\right),\mathbb{P}\right)$
from \thmref{gmH}, each are candidates for realization of \emph{wavelet
filters}. The function $m_{0}$ in (\ref{eq:uw1}) below is an example
of a wavelet filter; see also (\ref{eq:wa1}) above. 

It is known (see, e.g., \cite{MR1913212}) that a given wavelet filter
$m_{0}\left(t\right)$ generally does not admit a solution $\varphi$
in $L^{2}\left(\mathbb{R}\right)$. By this we mean that eq. (\ref{eq:wa1}),
or equivalently eq. (\ref{eq:uw2}), does \emph{not} have a solution
$\hat{\varphi}$ in $L^{2}\left(\mathbb{R}\right)$. 

The sub-class of wavelet filters which do admit $L^{2}\left(\mathbb{R}\right)$-solutions
is known to constitute only a ``small'' subset of all possible systems
of multi-band filters.

\thmref{gmH} shows: (i) that there are always wavelet solutions when
we resort to $L^{2}\left(Sol_{N}\left(\mathbb{T}\right),\mathbb{P}\right)$,
and (ii) \propref{uw} shows that, when $L^{2}\left(\mathbb{R}\right)$-solutions
$\varphi$ exist, then they automatically yield \emph{isometric} inclusions
$L^{2}\left(\mathbb{R}\right)\hookrightarrow L^{2}\left(Sol_{N}\left(\mathbb{T}\right),\mathbb{P}\right)$
(see \cite{MR0253059}).

We now turn to the link between the cases $L^{2}\left(\mathbb{R}\right)$
and $L^{2}\left(Sol_{N},\mathscr{C},\mathbb{P}\right)$ for the special
case where an $L^{2}\left(\mathbb{R}\right)$ wavelet exists as specified
in (\ref{eq:wa1})\textendash (\ref{eq:wa2}) above in \subsecref{gmmr}.

Let $\varphi$ be a choice of scaling function, see (\ref{eq:wa1}),
and let 
\begin{equation}
m_{0}\left(t\right):=\sum_{k\in\mathbb{Z}}a_{k}e^{i2\pi kt}.\label{eq:uw1}
\end{equation}
Then (see \cite{MR1913212,MR3347450})
\begin{equation}
\hat{\varphi}\left(t\right)=\frac{1}{\sqrt{N}}m_{0}\left(\frac{t}{N}\right)\hat{\varphi}\left(\frac{t}{N}\right),\;t\in\mathbb{R},\label{eq:uw2}
\end{equation}
where $\hat{\varphi}$ denotes the $L^{2}\left(\mathbb{R}\right)$-Fourier
transform. Set 
\begin{align}
\left(R_{m_{0}}f\right)\left(t\right) & =\frac{1}{N}\sum_{Ns=t\text{ mod 1}}\left|m_{0}\left(s\right)\right|^{2}f\left(s\right)\label{eq:uw3}\\
 & =\frac{1}{N}\sum_{k=0}^{N-1}\left(\left|m_{0}\right|^{2}f\right)\left(\frac{t+k}{N}\right),\;t\in\mathbb{T}=\mathbb{R}/\mathbb{Z},\nonumber 
\end{align}
and 
\begin{equation}
h_{\varphi}\left(t\right):=\sum_{n\in\mathbb{Z}}\left|\hat{\varphi}\left(t+n\right)\right|^{2},\label{eq:uw4}
\end{equation}
then 
\begin{equation}
R_{m_{0}}\left(h_{\varphi}\right)=h_{\varphi}.\label{eq:uw5}
\end{equation}

\begin{prop}
\label{prop:uw}Let $\varphi$, $m_{0}$, $R_{m_{0}}$, and $h_{\varphi}$
be as above. For 1-periodic functions $f$, i.e., $f$ on $\mathbb{R}/\mathbb{Z}$,
set 
\begin{equation}
L^{2}\left(\mathbb{R}\right)\ni\underset{\mathcal{V}_{0}}{\underbrace{f\left(t\right)\hat{\varphi}\left(t\right)}}\xmapsto{\;K_{0}\;}f\circ\pi_{0}\in\mathscr{H}_{0}\subset L^{2}\left(Sol_{N},\mathbb{P}\right)\label{eq:uw6}
\end{equation}
(where we use the construction of a multiresolution in $L^{2}\left(Sol_{N},\mathbb{P}\right)$
from \subsecref{mul}.) Then $K_{0}$ in (\ref{eq:uw6}) is \uline{isometric},
and it extends to become an isometry mapping $L^{2}\left(\mathbb{R}\right)$
into $L^{2}\left(Sol_{N},\mathbb{P}\right)$. 
\end{prop}
\begin{proof}
By \thmref{gmH}, we only need to check that $K_{0}$ is isometric
on the resolution subspace $\mathcal{V}_{0}\subset L^{2}\left(\mathbb{R}\right)$.
This follows from the computation:
\begin{align*}
\int_{\mathbb{R}}\left|f\left(t\right)\hat{\varphi}\left(t\right)\right|^{2}dt & =\int_{0}^{1}\left|f\left(t\right)\right|^{2}\sum_{n\in\mathbb{Z}}\left|\varphi\left(t+n\right)\right|^{2}dt\\
 & =\int_{0}^{1}\left|f\left(t\right)\right|^{2}h_{\varphi}\left(t\right)dt=\left\Vert f\circ\pi_{0}\right\Vert _{L^{2}\left(Sol_{N}\left(\mathbb{T}\right),\mathbb{P}\right)}^{2}.
\end{align*}
\end{proof}

\section{\label{sec:te}Two examples}

In this section we discuss two examples which serve to illustrate
the main results so far in Sections \ref{sec:ms}\textendash \ref{sec:gm}.
\begin{example}
\label{exa:2m1}$X=\mathbb{R}/\mathbb{Z}\simeq[0,1)$ with the usual
Borel sigma-algebra. Let $\sigma\left(x\right)=2x$ mod 1 (\figref{2mod1}),
and 
\[
\left(Rf\right)\left(x\right)=\frac{1}{2}\Bigl(f\Bigl(\frac{x}{2}\Bigr)+f\Bigl(\frac{x+1}{2}\Bigr)\Bigr).
\]
\end{example}

\begin{example}[See \figref{imp1}]
\label{exa:2m2} Let $X=\mathbb{R}/\mathbb{Z}\simeq[0,1)$, $\sigma\left(x\right)=2x$
mod 1, and 
\[
R\left(f\right)\left(x\right)=\cos^{2}\Bigl(\frac{\pi x}{2}\Bigr)f\Bigl(\frac{x}{2}\Bigr)+\sin^{2}\Bigl(\frac{\pi x}{2}\Bigr)f\Bigl(\frac{x+1}{2}\Bigr).
\]
Let $\lambda$ be the Lebesgue measure on $[0,1)$. In this case,
we have $\lambda\in Fix\left(\sigma\right)\cap\mathscr{L}\left(R\right)$,
but $\lambda\notin\mathscr{K}_{1}$. 
\end{example}
We shall return to these two examples in both \secref{K1} and \secref{AE}
below.

\begin{figure}[H]
\includegraphics[width=0.2\textwidth]{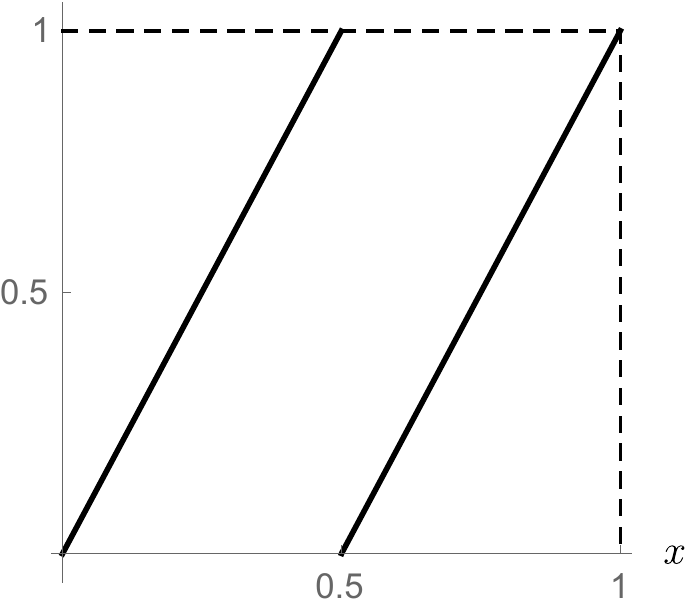}

\caption{\label{fig:2mod1}$\sigma\left(x\right)=2x$ mod 1}
\end{figure}

\begin{figure}[h]
\includegraphics[width=0.5\textwidth]{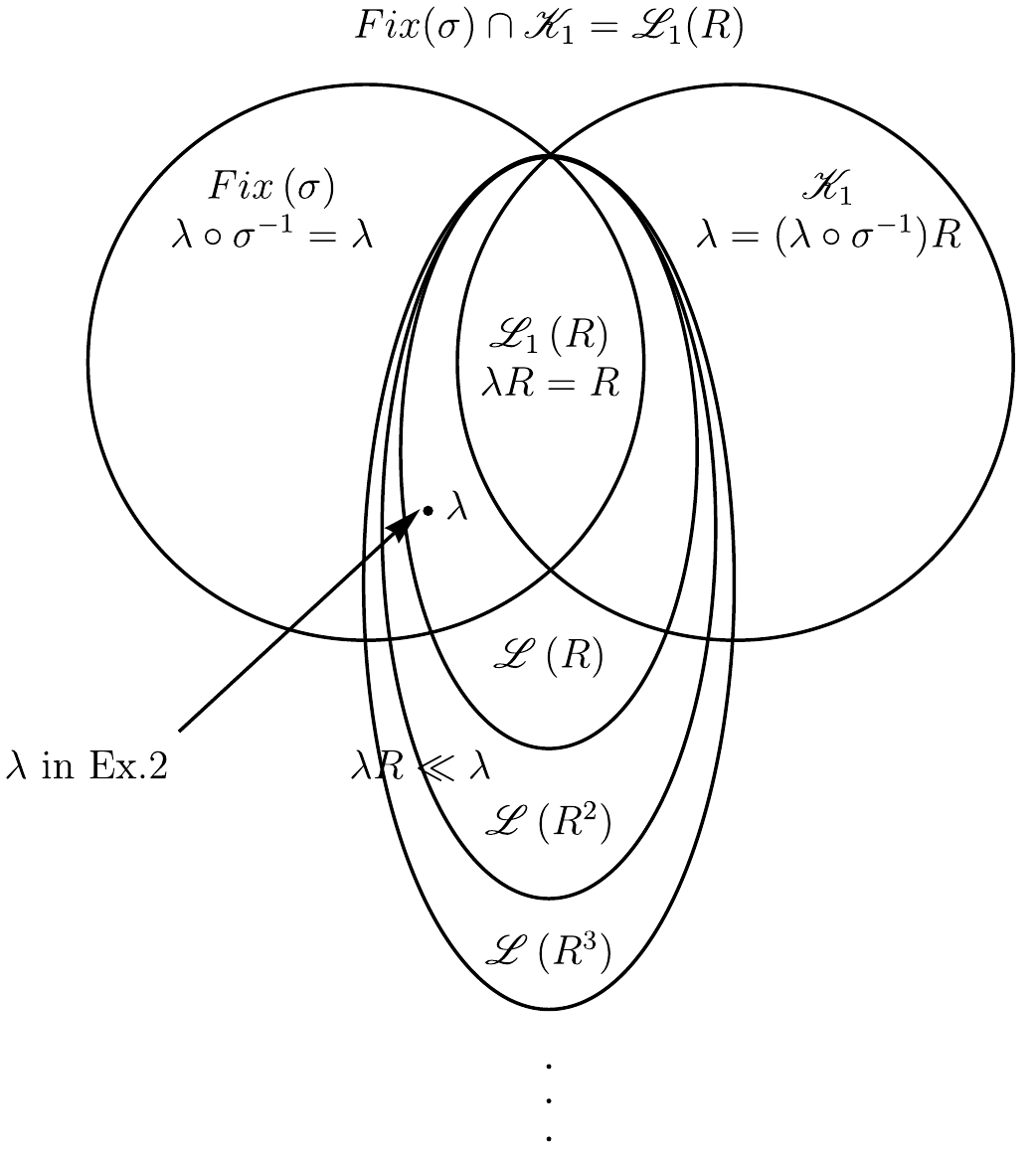}

\caption{\label{fig:imp1}Implications and containments. The containments and
intersections hold for the sets of measures associated to $\left(X,\mathscr{B},\sigma,R\right)$.
Note that in \exaref{2m2}, $d\lambda=$ Lebesgue measure, $\sigma\left(x\right)=2x$
mod 1; $\lambda\in Fix\left(\sigma\right)\cap\mathscr{L}\left(R\right)$,
but $\lambda\protect\notin\mathscr{K}_{1}$. For the various sets
referenced in the figure, we refer to \defref{meas} and \lemref{cm}
above.}
\end{figure}

\section{\label{sec:K1}The set $\mathscr{K}_{1}\left(X,\mathscr{B}\right)$}

Starting with an endomorphism of a measure space $\left(X,\mathscr{B}\right)$,
and a transfer operator $R$ (see, e.g., \cite{MR3124323,MR3459161,MR3375595,MR2176941,MR2129258}),
we study in the present section an associated family of convex set
of measures on $X$ (see \defref{meas} and \ref{lem:m1}) which yield
$R$-regular conditional expectations for the corresponding path-space
measure space $\left(\Omega_{X},\mathscr{C},\mathbb{P}\right)$. 
\begin{lem}
Let $\lambda\in\mathscr{K}_{1}$, then 
\[
\lambda\circ\sigma^{-1}\in\mathscr{L}\left(R\right)\Longleftrightarrow\lambda\ll\lambda\circ\sigma^{-1}.
\]
\end{lem}
\begin{proof}
Assume $\lambda\in\mathscr{K}_{1}$, and $\left(\lambda\circ\sigma^{-1}\right)R\ll\lambda\circ\sigma^{-1}$.
Since $\left(\lambda\circ\sigma^{-1}\right)R=\lambda$, we get $\lambda\ll\lambda\circ\sigma^{-1}$. 

Conversely, suppose $\lambda\ll\lambda\circ\sigma^{-1}$ and $\lambda=\left(\lambda\circ\sigma^{-1}\right)R$.
Then we conclude that $\lambda\circ\sigma^{-1}\in\mathscr{L}\left(R\right)$. 
\end{proof}

\begin{thm}
\label{thm:K1}Let $\left(X,\mathscr{B},\sigma,R\right)$ be as usual,
assuming $R\mathbbm{1}=\mathbbm{1}$. Suppose $\lambda\in\mathscr{L}\left(R\right)$,
and let $W=d\left(\lambda R\right)/d\lambda$. 

Then, $\lambda\in\mathscr{K}_{1}$ (so $\lambda\in\mathscr{K}_{1}\cap\mathscr{L}\left(R\right)$)
$\Longleftrightarrow$ $W\sim\sigma^{-1}\left(\mathscr{B}\right)$,
i.e., $W$ is measurable w.r.t the smaller sigma-algebra $\sigma^{-1}\left(\mathscr{B}\right)$. 
\end{thm}
\begin{proof}
Set $\nu=\lambda\circ\sigma^{-1}$, and $Q=d\nu/d\lambda$. We show
that 
\[
\lambda\in\mathscr{K}_{1}\Longleftrightarrow\nu R=\lambda\Longleftrightarrow\left(Q\circ\sigma\right)W=1\;\mbox{a.e. }\lambda.
\]
(Note that $\lambda\in\mathscr{K}_{1}\Longleftrightarrow\nu R=\lambda$,
see (\ref{eq:a3}).)

Now compute:
\begin{align*}
\int R\left(f\right)d\nu & =\int R\left(f\right)Q\,d\lambda=\int R\left(f\:\left(Q\circ\sigma\right)\right)d\lambda\\
 & =\int f\,\left(Q\circ\sigma\right)d\left(\lambda R\right)=\int f\,\left(Q\circ\sigma\right)W\,d\lambda,
\end{align*}
and it follows that $\nu R=\lambda\Longleftrightarrow\left(Q\circ\sigma\right)W=1$
a.e. $\lambda$. We need to find a solution $Q$ to 
\[
\left(Q\circ\sigma\right)\left(x\right)=\begin{cases}
\dfrac{1}{W\left(x\right)} & \mbox{if \ensuremath{W\left(x\right)\neq0}}\\
0 & \mbox{if \ensuremath{W\left(x\right)=0}}
\end{cases}
\]
which is equivalent to $W\sim\sigma^{-1}\left(\mathscr{B}\right)\Longleftrightarrow W^{-1}$
is $\sigma^{-1}\left(\mathscr{B}\right)$-measurable.
\end{proof}
\thmref{K1} can be restated as follows:
\begin{cor}
Suppose $\lambda\in\mathscr{L}\left(R\right)$ with $d\left(\lambda R\right)/d\lambda=W$,
then $\lambda\in\mathscr{K}_{1}\left(\cap\mathscr{L}\left(R\right)\right)$
$\Longleftrightarrow$ $\mathbb{E}^{\left(\lambda\right)}\left(W\bigm|{}_{\sigma^{-1}\left(\mathscr{B}\right)}\right)=W$,
i.e., $W\sim\sigma^{-1}\left(\mathscr{B}\right)$; but the measure
$\nu:=\lambda\circ\sigma^{-1}$ may be unbounded. 
\end{cor}
\begin{rem}
In general, the solution $\nu$ to $\lambda=\nu R$ may be an unbounded
measure.
\end{rem}
\renewcommand{\arraystretch}{1.5}
\noindent \begin{flushleft}
\begin{table}[H]
\noindent \begin{raggedright}
\begin{tabular}{|c|>{\centering}p{0.1\textwidth}|>{\centering}p{0.1\textwidth}|>{\centering}p{0.14\textwidth}|>{\centering}p{0.1\textwidth}|c|>{\centering}p{0.15\textwidth}|}
\hline 
Meas. & $\mathscr{L}\left(R\right)$ & $\mathscr{L}_{1}\left(R\right)$ & $Fix\left(\sigma\right)$ & $\mathscr{K}_{1}=M_{1}R$ & $\sqrt{\lambda}\in\mathscr{H}_{\infty}$ & $\underset{{\displaystyle \cap_{i}\mathscr{H}\left(\lambda R^{i}\right)}}{\sqrt{\lambda}\in}$\tabularnewline
\hline 
Defn. & $\lambda R\ll\lambda$ & $\lambda R=\lambda$ & $\lambda=\lambda\circ\sigma^{-1}$ & $\lambda=\nu R$ & $\widehat{S}\sqrt{\lambda}=\sqrt{\lambda}$ & \tabularnewline
\hline 
Ex \ref{exa:2m1}  & all $\lambda$ s.t. $\lambda\ll dx$ $\lambda_{1}=dx$ & (1) $\lambda_{1}=dx$ & $\lambda_{1}=dx$ & Ex \ref{exa:2m1} \\
$\lambda_{1}=dx$ $\lambda=\lambda R$ & $\lambda_{1}=dx$ & Ex \ref{exa:2m1} \\
$\lambda_{1}=dx$\tabularnewline
\hline 
Ex \ref{exa:2m2} & $\delta_{0}$, $\lambda_{1}=dx$ & (2) $\delta_{0}$, singletons & $\delta_{0}$, $\lambda=dx$ & $\delta_{0}$\\
Ex \ref{exa:2m2} $\lambda\notin\mathscr{K}_{1}$ & $\delta_{0}$ & Ex \ref{exa:2m2} \\
If $\lambda=dx$, then $\cap_{i}\mathscr{H}\left(\lambda R^{i}\right)$\\
$=0$\tabularnewline
\hline 
\end{tabular}
\par\end{raggedright}
\caption{Illustration by Examples. The set of measures itemized in the first
two lines of the table refer to the operator $R$ as given in the
two examples, Examples \ref{exa:2m1} (line 3), and \ref{exa:2m2}
(line 4.) The verification of the respective properties is left to
the reader.}
\end{table}
\par\end{flushleft}

\renewcommand{\arraystretch}{1}

\section{The universal Hilbert space}

Starting with an endomorphism $\sigma$ of a measure space $X$, and
a transfer operator $R$, we study in the present section a certain
universal Hilbert space which allows an operator realization of the
pair $\left(\sigma,R\right)$.

We refer to this as a universal Hilbert space as it involves equivalence
classes defined from all possible measures on a fixed measure space,
see e.g., \cite{MR0282379}. Because of work by \cite{MR3394108,MR2240643,MR2097020}
it is also known that this Hilbert space has certain universality
properties.

We shall need the following Hilbert space $\mathscr{H}\left(X\right)$
of equivalence classes of pairs $\left(f,\lambda\right)$, $f\in\mathscr{F}\left(X,\mathscr{B}\right)$,
$\lambda\in M\left(X,\mathscr{B}\right)$ (= all Borel measures on
$\left(X,\mathscr{B}\right)$).
\begin{defn}
\label{def:uh1}Two pairs $\left(f,\lambda\right)$ and $\left(g,\mu\right)$
are said to be equivalent, $\left(f,\lambda\right)\sim\left(g,\mu\right)$,
iff (Def.) there exists $\xi$ s.t. $\lambda\ll\xi$, $\mu\ll\xi$,
and 
\[
f\sqrt{\frac{d\lambda}{d\xi}}=g\sqrt{\frac{d\mu}{d\xi}}\quad\mbox{a.e. }\xi.
\]
The equivalence class of $\left(f,\lambda\right)$ is denoted $f\sqrt{\lambda}$. 
\end{defn}

\begin{defn}
\label{def:uh2}Set 
\begin{align*}
\bigl\Vert f\sqrt{\lambda}\bigr\Vert_{\mathscr{H}\left(X\right)}^{2} & =\int_{X}\left|f\right|^{2}d\lambda,\;\mbox{and}\\
\left\langle f_{1}\sqrt{\lambda_{1}},f_{2}\sqrt{\lambda_{2}}\right\rangle _{\mathscr{H}\left(X\right)} & =\int_{X}\overline{f_{1}}f_{2}\sqrt{\frac{d\lambda_{1}}{d\mu}}\sqrt{\frac{d\lambda_{2}}{d\mu}}d\mu
\end{align*}
if $\lambda_{i}\ll\mu$, $i=1,2$. 
\end{defn}
\begin{lem}
\label{lem:uh1}Let $\left(X,\mathscr{B},\sigma,R\right)$ be as above,
assuming $R\mathbbm{1}=\mathbbm{1}$. Then the mapping 
\begin{equation}
\widehat{S}(f\sqrt{\lambda}):=\left(f\circ\sigma\right)\sqrt{\lambda R},\quad\forall f\sqrt{\lambda}\in\mathscr{H}\left(X\right),\label{eq:h1}
\end{equation}
is well defined and \uline{isometric}.
\end{lem}

\begin{proof}
A direct verification shows that $\widehat{S}$ is well defined. Now
we show that $\Vert\widehat{S}v\Vert_{\mathscr{H}\left(X\right)}=\left\Vert v\right\Vert _{\mathscr{H}\left(X\right)}$,
$\forall v\in\mathscr{H}\left(X\right)$. Setting $v=f\sqrt{\lambda}$,
we must show that
\begin{equation}
\bigl\Vert f\sqrt{\lambda}\bigr\Vert_{\mathscr{H}\left(X\right)}^{2}=\bigl\Vert\left(f\circ\sigma\right)\sqrt{\lambda R}\bigr\Vert_{\mathscr{H}\left(X\right)}^{2}.\label{eq:h2}
\end{equation}
Note that 
\begin{align*}
\mbox{RHS}_{\left(\ref{eq:h2}\right)} & =\int_{X}f^{2}\circ\sigma\,d\left(\lambda R\right)=\int_{X}R\left(f^{2}\circ\sigma\right)d\lambda\\
 & =\int_{X}f^{2}\underset{=\mathbbm{1}}{\underbrace{R\mathbbm{1}}}d\lambda=\int_{X}f^{2}d\lambda=\mbox{LHS}_{\left(\ref{eq:h2}\right)}.
\end{align*}
\end{proof}
\begin{rem}
\lemref{uh1} yields the Wold decomposition of $\mathscr{H}\left(X\right)$:
\[
\mathscr{H}\left(X\right)=\left(\mbox{Wold shift}\right)\oplus\mathscr{H}_{\infty}
\]
where $\mathscr{H}_{\infty}$ denotes the unitary part. See, e.g.,
\cite{MR1913212,MR2641594,MR1711343,MR616145}.
\end{rem}
Below we outline the operator theoretic details entailed in the analysis
in our universal Hilbert space.
\begin{lem}
Set 
\begin{equation}
\mathscr{H}\left(\mathscr{K}_{1}\right)=\left\{ f\sqrt{\lambda}\in\mathscr{H}\left(X\right)\mid\lambda\in\mathscr{K}_{1}\right\} \label{eq:h3}
\end{equation}
where $\mathscr{K}_{1}=M_{1}R$ (see \lemref{m1}). Then $\mathscr{H}\left(\mathscr{K}_{1}\right)\subset\mathscr{H}\left(X\right)$
is a closed subspace. 
\end{lem}
\begin{defn}
Let $P_{\mathscr{K}}$ be the orthogonal projection onto $\mathscr{H}\left(\mathscr{K}_{1}\right)$. 
\end{defn}
\begin{lem}
\label{lem:Rh}Let $\widehat{S}$ be as in (\ref{eq:h1})\@. Set
\begin{equation}
\widehat{R}\left(g\sqrt{\mu}\right)=R\left(g\right)\sqrt{\mu_{\mathscr{K}}\circ\sigma^{-1}},\quad\sqrt{\mu_{\mathscr{K}}}:=P_{\mathscr{K}}\sqrt{\mu};\label{eq:h4}
\end{equation}
then $\widehat{S}$, $\widehat{R}$ form a symmetric pair in $\mathscr{H}\left(X\right)$,
\begin{equation}
\left\langle \widehat{S}v,w\right\rangle _{\mathscr{H}\left(X\right)}=\left\langle v,\widehat{R}w\right\rangle _{\mathscr{H}\left(X\right)},\quad\forall v,w\in\mathscr{H}\left(X\right).\label{eq:h5}
\end{equation}
That is, 
\begin{equation}
\widehat{R}=\widehat{S}^{*}.\label{eq:h4a}
\end{equation}
\end{lem}
\begin{proof}
We note that (\ref{eq:h5}) $\Longleftrightarrow$ 
\begin{equation}
\left\langle f\circ\sigma\sqrt{\lambda R},g\sqrt{\mu}\right\rangle _{\mathscr{H}\left(X\right)}=\left\langle f\sqrt{\lambda},R\left(g\right)\sqrt{\mu_{\mathscr{K}}\circ\sigma^{-1}}\right\rangle _{\mathscr{H}\left(X\right)},\label{eq:h6}
\end{equation}
$\forall f\sqrt{\lambda},g\sqrt{\mu}\in\mathscr{H}\left(X\right)$. 

To verify (\ref{eq:h6}): 
\[
\mbox{RHS}_{\left(\ref{eq:h6}\right)}=\int_{X}fR\left(g\right)\sqrt{\frac{d\lambda}{d\xi}\frac{d\mu_{\mathscr{K}}\circ\sigma^{-1}}{d\xi}}d\xi
\]
and 
\begin{align*}
\mbox{LHS}_{\left(\ref{eq:h6}\right)} & =\int_{X}\left(f\circ\sigma\right)g\sqrt{\frac{d\left(\lambda R\right)}{d\left(\xi R\right)}\cdot\frac{d\mu}{d\left(\xi R\right)}}d\left(\xi R\right)\\
 & =\int_{X}\left(f\circ\sigma\right)g\sqrt{\left(\frac{d\lambda}{d\xi}\right)\circ\sigma\cdot\left(\frac{d\mu_{\mathscr{K}}\circ\sigma^{-1}}{d\xi}\right)\circ\sigma}\,d\left(\xi R\right)\\
 & =\int_{X}fR\left(g\right)\sqrt{\left(\frac{d\lambda}{d\xi}\right)\frac{d\mu_{\mathscr{K}}\circ\sigma^{-1}}{d\xi}}d\xi=\mbox{RHS}_{\left(\ref{eq:h6}\right)},
\end{align*}
where we used the following substitution rules (see \lemref{abs1})
\begin{align*}
\frac{d\left(\lambda R\right)}{d\left(\xi R\right)} & =\frac{d\lambda}{d\xi}\circ\sigma\\
\frac{d\mu}{d\left(\xi R\right)} & =\left(\frac{d\mu_{\mathscr{K}}\circ\sigma^{-1}}{d\xi}\right)\circ\sigma
\end{align*}
for the respective Radon-Nikodym derivatives. 

Note that we also used that 
\[
\left[\begin{array}{c}
\lambda\ll\xi\\
\mu_{\mathscr{K}}\circ\sigma^{-1}\ll\xi
\end{array}\right]\Longrightarrow\left[\begin{array}{c}
\lambda R\ll\xi R\\
\mu\ll\xi R
\end{array}\right].
\]
\end{proof}

\begin{cor}
\label{cor:O1}Given $\left(X,\mathscr{B},\sigma,R\right)$, $R\mathbbm{1}=\mathbbm{1}$,
as introduced above. Let $\widehat{S}$, $\widehat{R}=\widehat{S}^{*}$
be the canonical operators in $\mathscr{H}\left(X\right)$, then 
\begin{enumerate}
\item \label{enu:h1}$\widehat{R}\widehat{S}=\widehat{S}^{*}\widehat{S}=I_{\mathscr{H}\left(X\right)}$
; 
\item \label{enu:h2}$\widehat{S}\widehat{R}=\widehat{S}\widehat{S}^{*}=\widehat{E}_{1}$
= the projection onto $\widehat{S}\mathscr{H}\left(X\right)$; and
\item \label{enu:h3}$\widehat{E}_{1}\bigl(f\sqrt{\lambda}\bigr)=R\left(f\right)\circ\sigma\sqrt{\lambda_{\mathscr{K}}}$,
where $\sqrt{\lambda_{\mathscr{K}}}=P_{\mathscr{K}}\sqrt{\lambda}$. 
\end{enumerate}
\end{cor}
\begin{proof}
We already proved (\ref{enu:h1})-(\ref{enu:h2}); recall that 
\[
f\sqrt{\lambda}\xrightarrow{\;\widehat{S}\;}f\circ\sigma\sqrt{\lambda R}\xrightarrow{\;\widehat{R}\;}R\left(f\circ\sigma\right)\sqrt{\lambda R\circ\sigma^{-1}}=f\sqrt{\lambda}.
\]
Proof of (\ref{enu:h3}). 
\begin{align*}
\widehat{E}_{1}\left(f\sqrt{\lambda}\right) & =\widehat{S}\widehat{R}f\sqrt{\lambda}=\widehat{S}\left(R\left(f\right)\sqrt{\lambda_{\mathscr{K}}\circ\sigma^{-1}}\right)\\
 & =R\left(f\right)\circ\sigma\sqrt{\lambda_{\mathscr{K}}\circ\sigma^{-1}R}=R\left(f\right)\circ\sigma\sqrt{\lambda_{\mathscr{K}}}.
\end{align*}
In the last step we used that $\lambda_{\mathscr{K}}\in\mathscr{K}_{1}$
s.t. $\left(\lambda_{\mathscr{K}}\circ\sigma^{-1}\right)R=\lambda_{\mathscr{K}}$,
and the conditional expectation on $\sigma^{-1}\left(\mathscr{B}\right)$,
i.e., $\mathbb{E}^{\left(\lambda_{\mathscr{K}}\right)}\left(f\mid\sigma^{-1}\left(\mathscr{B}\right)\right)$;
see \defref{gmce}.
\end{proof}
\begin{question}
In \exaref{2m2} with $\lambda=dx=$Lebesgue, what is $\lambda_{\mathscr{K}}$,
i.e., $\sqrt{\lambda_{\mathscr{K}}}=\mbox{Proj}_{\mathscr{K}_{1}}(\sqrt{\lambda})$?
See \remref{h1} below.
\end{question}

\begin{lem}
We can establish the increasing sets 
\begin{equation}
\mathscr{L}\left(R\right)\subseteq\mathscr{L}\left(R^{2}\right)\subseteq\mathscr{L}\left(R^{3}\right)\subseteq\cdots\label{eq:j1}
\end{equation}
as follows:
\begin{equation}
\mathscr{L}\left(R\right)\xhookrightarrow{\quad\widehat{R}\quad}\mathscr{L}\left(R^{2}\right)\xhookrightarrow{\quad\widehat{R}\quad}\mathscr{L}\left(R^{4}\right)\xhookrightarrow{\quad\widehat{R}\quad}\cdots\label{eq:j2}
\end{equation}
\end{lem}
\begin{proof}
For (\ref{eq:j2}), since $\lambda\in\mathscr{L}\left(R\right)$,
$\lambda R\ll\lambda$, and $\widehat{R}\sqrt{\lambda}=\sqrt{\lambda_{\mathscr{K}}\circ\sigma^{-1}}$,
so $\widehat{R}\mathscr{L}\left(R\right)\subset\mathscr{L}\left(R^{2}\right)$
as 
\[
\left(\lambda_{\mathscr{K}}\circ\sigma^{-1}\right)R^{2}=\underset{\lambda_{\mathscr{K}}}{\underbrace{\left(\lambda_{\mathscr{K}}\circ\sigma^{-1}\right)R}}R=\lambda_{\mathscr{K}}R\ll\lambda.
\]

(Note the containment in (\ref{eq:j1}) refers to measure classes
in $\mathscr{H}\left(X\right)$. See also \thmref{meas1}.) 
\end{proof}
\begin{rem}
\label{rem:h1}Let $\left(X,\mathscr{B},\sigma,R\right)$, $R\mathbbm{1}=\mathbbm{1}$
be as usual, and let $\widehat{S}$, and $\widehat{R}=\widehat{S}^{*}$
be the universal operators; see (\ref{eq:h1}) and (\ref{eq:h4}). 

If, in addition, $\lambda\in\mathscr{L}\left(R\right)$ with $d\left(\lambda R\right)/d\lambda=W$,
then we also have 
\begin{equation}
\widehat{R}\left(f\sqrt{\lambda}\right)=R\left(\frac{f}{\sqrt{W}}\right)\sqrt{\lambda}.\label{eq:j4}
\end{equation}
\end{rem}
\begin{proof}
Eq (\ref{eq:j4}) is verified as follows:
\begin{align*}
\left\langle f\circ\sigma\sqrt{\lambda R},g\sqrt{\lambda}\right\rangle _{\mathscr{H}\left(X\right)} & =\int_{X}\left(f\circ\sigma\right)g\sqrt{W}d\lambda=\int_{X}\frac{\left(f\circ\sigma\right)g}{\sqrt{W}}d\left(\lambda R\right)\\
 & =\int_{X}fR\left(\frac{g}{\sqrt{W}}\right)d\lambda=\left\langle f\sqrt{\lambda},R\left(\frac{g}{\sqrt{W}}\sqrt{\lambda}\right)\right\rangle _{\mathscr{H}\left(X\right)}.
\end{align*}
\end{proof}

\begin{cor}
Suppose $\lambda\in\mathscr{L}\left(R\right)$, $d\left(\lambda R\right)/d\lambda=W$,
then 
\begin{equation}
\lambda\ll\lambda_{\mathscr{K}}\circ\sigma^{-1},\;\mbox{and}\quad\frac{d\lambda}{d\left(\lambda_{\mathscr{K}}\circ\sigma^{-1}\right)}=W\circ\sigma^{-1}.\label{eq:j5}
\end{equation}
\end{cor}

\begin{proof}
From \remref{h1} we have 
\begin{equation}
\left(R\left(f\right),\lambda_{\mathscr{K}}\circ\sigma^{-1}\right)\sim\left(R\left(\frac{f}{\sqrt{W}}\right),\lambda\right)\label{eq:j6}
\end{equation}
and since $\lambda R\ll\lambda$, we get $\lambda\ll\lambda_{\mathscr{K}}\circ\sigma^{-1}$.
So (\ref{eq:j6}) $\Longrightarrow$
\[
R\left(f\right)=R\left(\frac{f}{\sqrt{W}}\right)\sqrt{\frac{d\lambda}{d\left(\lambda_{\mathscr{K}}\circ\sigma^{-1}\right)}}=R\left(\left(\frac{f}{\sqrt{W}}\right)\sqrt{W}\right),
\]
and (\ref{eq:j5}) follows. 
\end{proof}
\begin{cor}
Suppose $\lambda\in\mathscr{L}\left(R\right)$, $d\left(\lambda R\right)/d\lambda=W$,
then
\[
P_{\mathscr{K}}\sqrt{\lambda}=\frac{1}{\sqrt{W}}\sqrt{\lambda R}.
\]
\end{cor}
\begin{proof}
Follows from \remref{h1}. 
\end{proof}
\begin{lem}
Let $\left(X,\mathscr{B},\sigma,R\right)$ be as above, $R\mathbbm{1}=\mathbbm{1}$.
Suppose $\mu R\ll\mu$, $\frac{d\mu R}{d\mu}=W$. Then
\begin{enumerate}
\item \label{enu:h4}$S:f\longrightarrow f\circ\sigma\sqrt{W}$ is \uline{isometric}
in $L^{2}\left(\mu\right)$; and 
\item \label{enu:h5}we have $R^{*}:f\longrightarrow\left(f\circ\sigma\right)W$
in $L^{2}\left(\mu\right)$. 
\end{enumerate}
\end{lem}
\begin{proof}
We check that

(\ref{enu:h4}) 
\[
\int\left(\left(f\circ\sigma\right)\sqrt{W}\right)^{2}d\mu=\int f^{2}\circ\sigma Wd\mu=\int R\left(f^{2}\circ\sigma\right)d\mu=\int f^{2}d\mu;
\]

(\ref{enu:h5}) 
\[
\int\left(f\circ\sigma\right)Wgd\mu=\int\left(f\circ\sigma\right)g\,d\mu R=\int R\left(\left(f\circ\sigma\right)g\right)d\mu=\int fR\left(g\right)d\mu.
\]
\end{proof}

\section{\label{sec:el}Ergodic limits}

We now turn to a number of ergodic theoretic results that are feasible
in the general setting of pairs $\left(\sigma,R\right)$. See, e.g.,
\cite{MR617913}, and also Definitions \ref{def:meas}, \ref{def:uh1}
and Lemmas \ref{lem:uh1}, \ref{lem:Rh}. 
\begin{thm}
Given $\left(X,\mathscr{B},\sigma,R\right)$, $R\mathbbm{1}=\mathbbm{1}$,
as usual; then the following two conditions are equivalent:

\[
\bigcap_{i}\mathscr{L}\left(\lambda R^{i}\right)\neq0
\]
\[
\Updownarrow
\]
\begin{equation}
\lim_{N\rightarrow\infty}\underset{=:A_{N}}{\underbrace{\frac{1}{N}\sum_{k=1}^{N}\prod_{j=0}^{k-1}\sqrt{W\circ\sigma^{j}}}}\neq0\;\mbox{in }L^{2}\left(\lambda\right),\;\mbox{i.e.,}\label{eq:e1}
\end{equation}
\begin{gather*}
\lim_{N\rightarrow\infty}\frac{1}{N}\sum_{k=1}^{n}\sqrt{W\left(W\circ\sigma\right)\cdots\left(W\circ\sigma^{k-1}\right)}=W_{\infty}\in L^{2}\left(\lambda\right),\;\mbox{where}\\
W_{\infty}\neq0,\;\mbox{and }\int W_{\infty}d\lambda\leq1.
\end{gather*}

\end{thm}
\begin{rem}
Existence of the limit in (\ref{eq:e1}) is automatic.
\end{rem}
\begin{proof}
Return to $\mathscr{H}\left(X\right)$, $\widehat{S}$ and $\widehat{R}$,
and suppose 
\begin{equation}
\mu R\ll\mu;\label{eq:e6}
\end{equation}
so 
\begin{equation}
\xymatrix{\mathscr{H}\left(\mu R^{i}\right)\hookrightarrow\ar@{--}[r] & \mathscr{H}\left(\mu R^{2}\right)\ar@{^{(}->}[r] & \mathscr{H}\left(\mu R\right)\ar@{^{(}->}[r]\ar@/_{1.3pc}/[l]_{\widehat{S}} & \mathscr{H}\left(\mu\right)\ar@/_{1.3pc}/[l]_{\widehat{S}}}
\label{eq:e7}
\end{equation}
where we used:
\begin{equation}
\widehat{S}\left(f\sqrt{\lambda}\right):=\left(f\circ\sigma\right)\sqrt{\lambda R},\quad\forall f\sqrt{\lambda}\in\mathscr{H}\left(X\right).\label{eq:e8}
\end{equation}

We note that 
\[
\frac{1}{N}\sum_{k=1}^{N}\widehat{S}^{k}\longrightarrow\widehat{E}_{1}
\]
where $\widehat{E}_{1}$ is the projection in $\mathscr{H}\left(X\right)$
onto $\left\{ w_{1}\in\mathscr{H}\left(X\right)\mid\widehat{S}w_{1}=w_{1}\right\} $.
This is a version of von Neumann's ergodic theorem; see e.g., \cite{MR617913}.
Thus 
\begin{equation}
\lim_{N\rightarrow\infty}\left\Vert \widehat{A}_{N}\sqrt{\mu}-\widehat{E}_{1}\sqrt{\mu}\right\Vert _{\mathscr{H}\left(X\right)}=0.\label{eq:e9}
\end{equation}
\end{proof}

\begin{thm}
Let $\left(X,\mathscr{B},\sigma,R\right)$, $R\mathbbm{1}=\mathbbm{1}$,
be as above. Let $\lambda\in\mathscr{L}\left(R\right)$, $\frac{d\lambda R}{d\lambda}=W$,
and set 
\begin{equation}
A_{N}=\frac{1}{N}\sum_{k=1}^{N}\prod_{j=0}^{k-1}\sqrt{W\circ\sigma^{j}};\label{eq:f1}
\end{equation}
then $\exists W_{\infty}\in L^{2}\left(\lambda\right)$ s.t. 
\begin{equation}
A_{N}\xrightarrow[\;N\rightarrow\infty\;]{}W_{\infty}\;\mbox{pointwise }\lambda\mbox{-a.e.}\label{eq:f2}
\end{equation}
\end{thm}
\begin{proof}
We saw that the sequence in (\ref{eq:f1}) corresponds to the measure
ergodic limit $\lim_{N}\frac{1}{N}\sum_{k=1}^{N}\widehat{S}$ as an
operator limit in $\mathscr{H}\left(X\right)$ since $\widehat{S}$
is isometric. Hence (\ref{eq:f1}) has a subsequence which converges
$\lambda$ a.e. But since 
\begin{gather}
\frac{1}{N+1}\prod_{k=0}^{N}\sqrt{W\circ\sigma^{k}}\longrightarrow0,\;\mbox{and}\nonumber \\
A_{N+1}=\frac{N}{N+1}A_{N}+\frac{1}{N+1}\prod_{k=0}^{N}\sqrt{W\circ\sigma^{k}}\label{eq:f3}
\end{gather}
we conclude that $\left\{ A_{N}\right\} _{N\in\mathbb{Z}_{+}}$ converges
itself, $\lambda$-a.e., as $N\longrightarrow\infty$. 

Indeed, we assume $\lambda\in\mathscr{L}\left(R\right)$; set $W=\frac{d\lambda R}{d\lambda}$,
\begin{gather}
\left\Vert \frac{1}{N}\sum_{k=1}^{N}\widehat{S}^{k}\sqrt{\lambda}-W_{\infty}\sqrt{\lambda}\right\Vert _{\mathscr{H}\left(X\right)}\xrightarrow[\;N\rightarrow\infty\;]{}0\label{eq:f4}\\
\Updownarrow\nonumber \\
\left\Vert \frac{1}{N}\sum_{k=1}^{N}\prod_{j=0}^{k-1}\sqrt{W\circ\sigma^{j}}-W_{\infty}\right\Vert _{L^{2}\left(\lambda\right)}\xrightarrow[\;N\rightarrow\infty\;]{}0\label{eq:f5}
\end{gather}
and (\ref{eq:f3}) $\Leftrightarrow$ (\ref{eq:f4}) $\Leftrightarrow$
(\ref{eq:f5}). But by the von Neumann-Yosida ergodic theorem \cite{MR617913},
the limit in (\ref{eq:f4}) automatically exists. Hence $W_{\infty}\in L^{2}\left(\lambda\right)$,
the limit function may be zero; this holds in \exaref{2m2}, $\lambda=dx=$
Lebesgue measure, i.e., 
\begin{equation}
\frac{1}{N+1}\prod_{j=0}^{N}\sqrt{W\circ\sigma^{j}}\xrightarrow[\;N\rightarrow\infty\;]{}0.\label{eq:f6}
\end{equation}

Now it follows from the general ergodic theorem; the von Neumann-Yosida
theorem in the Hilbert space $\mathscr{H}\left(X\right)$, or in $\mathscr{H}\left(\mu\right)$,
that the limit \emph{exists}, where:
\begin{equation}
A_{N}\left(\sqrt{\mu}\right)\xrightarrow[\;N\rightarrow\infty\;]{}\widehat{E}_{1}\left(\sqrt{\mu}\right)\;\mbox{exists},
\end{equation}
but $\sqrt{\mu_{\infty}}$ may be zero (see \exaref{2m2}). 

We shall establish that the limit function 
\begin{equation}
\lim_{N\rightarrow\infty}\frac{1}{N}\sum_{k=1}^{N}\prod_{j=0}^{k-1}\sqrt{W\circ\sigma^{j}}=W_{\infty}
\end{equation}
is a non-zero function in $L^{2}\left(X,\mu\right)$. 

Suppose $0\leq W\leq1$ pointwise, then we get the formula
\[
A_{N}=\frac{1}{N}\sum_{k=1}^{N}\prod_{j=0}^{k-1}\sqrt{W\circ\sigma^{j}}
\]
and monotone decreasing as $N\longrightarrow\infty$: 
\begin{align*}
 & A_{N}-A_{N+1}\\
= & \frac{1}{N}\left[\sqrt{W}+\sqrt{W}\sqrt{W\circ\sigma}+\cdots+\sqrt{W}\sqrt{W\circ\sigma}\cdots\sqrt{W\circ\sigma^{N-1}}\right]\\
 & -\frac{1}{N+1}\left[\sqrt{W}+\sqrt{W}\sqrt{W\circ\sigma}+\cdots+\sqrt{W}\sqrt{W\circ\sigma}\cdots\sqrt{W\circ\sigma^{N}}\right]\\
= & \frac{1}{N+1}\left[\sqrt{W}+\sqrt{W}\sqrt{W\circ\sigma}+\cdots+\sqrt{W}\cdots\sqrt{W\circ\sigma^{N-1}}\right]\\
 & -\frac{1}{N+1}\sqrt{W}\cdots\sqrt{W\circ\sigma^{N}}\geq0.
\end{align*}
(\exaref{2m2} illustrates $A_{N}\longrightarrow0$ is possible, since
in this case $\bigcap_{i}\mathscr{H}\left(\lambda R^{i}\right)=0$,
where $\lambda=dx=$ Lebesgue measure.)

Since $\widehat{S}$ is an isometry in $\mathscr{H}\left(X\right)$,
we apply Wold's theorem \cite{MR1913212,MR2641594,MR1711343,MR616145}
to get existence of the limit in (\ref{eq:f6}), so $\exists\mu_{\infty}\in M_{+}$,
$\sqrt{\mu_{\infty}}\in\mathscr{H}\left(X\right)$ s.t. 
\[
\lim_{N\rightarrow\infty}\left\Vert \sqrt{\mu_{\infty}}-\frac{1}{N}\left(\sum_{k=1}^{N}\widehat{S}^{k}\right)\left(\sqrt{\lambda}\right)\right\Vert _{\mathscr{H}\left(X\right)}=0,
\]
and $\mu_{\infty}=\mu_{\infty}R$. (Note. $\mu_{\infty}\ll\lambda$,
but $\mu_{\infty}=0$ is possible. This is precisely what happens
in \exaref{2m2}.)

Pass to $\widehat{S}$, $\mathscr{H}\left(X\right)$, where $\widehat{S}\left(f\sqrt{\mu}\right)=\left(f\circ\sigma\right)\sqrt{\mu R}$,
$\forall f\sqrt{\mu}\in\mathscr{H}\left(X\right)$, then $\widehat{S}$
is isometric in $\mathscr{H}\left(X\right)$ (see \lemref{uh1}),
and if $\lambda\in\mathscr{L}\left(R\right)$ then $\widehat{S}$
restricts to an isometry in $\mathscr{H}\left(\lambda\right)\simeq L^{2}\left(\lambda\right)$
(unitary), and 
\[
\widehat{S}^{k}\sqrt{\lambda}=\sqrt{\lambda R^{k}}=\sqrt{W\left(W\circ\sigma\right)\cdots W\circ\sigma^{k-1}}\sqrt{\lambda},
\]
and by the general theorem (von Neumann and Yosida), $\frac{1}{N}\sum_{k=1}^{N}\widehat{S}^{k}\sqrt{\lambda}$
exists in $\mathscr{H}\left(\lambda\right)$. 
\end{proof}
\begin{lem}
Let $\left(X,\mathscr{B},\sigma,R\right)$, $R\mathbbm{1}=\mathbbm{1}$,
be as above. Suppose $\lambda R\ll\lambda$, $W=\frac{d\lambda R}{d\lambda}$,
and let $\mu_{\alpha}$ be a measure on $\left(X,\mathscr{B}\right)$
s.t. 
\begin{equation}
\lim_{N\rightarrow\infty}\left\Vert \sqrt{\mu_{\alpha}}-\frac{1}{N}\sum_{k=1}^{N}\widehat{S}^{k}\sqrt{\lambda}\right\Vert _{\mathscr{H}\left(X\right)}=0.\label{eq:g1}
\end{equation}
Recall $\widehat{S}^{k}\sqrt{\lambda}=\prod_{j=0}^{k-1}\sqrt{W\circ\sigma^{j}}\sqrt{\lambda}$,
and (\ref{eq:g1}) is equivalent to 
\begin{equation}
\int_{X}\left|W_{\infty}-A_{N}\right|^{2}d\lambda\xrightarrow[\;N\rightarrow\infty\;]{}0.\label{eq:g2}
\end{equation}

Then, $\mu_{\infty}R=\mu_{\alpha}$ and $\mu_{\infty}\ll\lambda$,
where $d\mu_{\infty}=W_{\infty}d\lambda$, $W_{\infty}\in L^{2}\left(\lambda\right)$,
and $\sqrt{\mu_{\infty}}\in\mathscr{H}\left(\lambda\right)$. However,
$\mu_{\infty}=0$ is possible.
\end{lem}
\begin{rem}
\exaref{2m2} shows that $\mu_{\infty}=0$ is possible.
\end{rem}
We do have a general condition:
\begin{prop}
\label{prop:Rp1}Assume $\lambda R\ll\lambda$, $W=\frac{d\lambda R}{d\lambda}$,
then 
\begin{gather}
\frac{1}{N}\sum_{k=1}^{N}\widehat{S}^{k}\sqrt{\lambda}\xrightarrow[\;N\rightarrow\infty\;]{}\sqrt{W_{\infty}\lambda}\label{eq:g3}\\
\Updownarrow\nonumber \\
\underset{=:A_{N}}{\underbrace{\frac{1}{N}\sum_{k=1}^{N}\prod_{j=0}^{k-1}\sqrt{W\circ\sigma^{j}}\sqrt{\lambda}}}\quad\mbox{has a limit in }L^{2}\left(\lambda\right).\label{eq:g4}
\end{gather}
\end{prop}
\begin{question}
Is it still possible that $\exists\mu_{\infty}\in M_{1}$ s.t. 
\begin{equation}
\lim_{N\rightarrow\infty}\frac{1}{N}\sum_{k=1}^{N}\widehat{S}^{k}\sqrt{\lambda}=\sqrt{\mu_{\infty}},\label{eq:g5}
\end{equation}
and $\mu_{\infty}R=\mu_{\infty}$, and $\mu_{\infty}\ll\lambda$?
\end{question}
\begin{proof}[Proof of \propref{Rp1}]
Note that $\lambda R\ll\lambda$ $\Longrightarrow$ $\widehat{S}$
in $\mathscr{H}\left(X\right)$. Indeed, $\widehat{S}\mathscr{H}\left(\lambda\right)\subset\mathscr{H}\left(\lambda\right)$,
which is closed in $\mathscr{H}\left(X\right)$. To see this, we check
that 
\[
\widehat{S}f\sqrt{\lambda}=\left(f\circ\sigma\right)\sqrt{\lambda R}=\left(f\circ\sigma\right)\sqrt{W}\sqrt{\lambda}\in\mathscr{H}\left(\lambda\right)
\]
and 
\[
\left\Vert \left(f\circ\sigma\right)\sqrt{W}\sqrt{\lambda}\right\Vert _{\mathscr{H}\left(\lambda\right)}^{2}=\int\left(f^{2}\circ\sigma\right)\underset{\lambda R}{\underbrace{Wd\lambda}}=\int f^{2}d\lambda=\left\Vert f\right\Vert _{L^{2}\left(\lambda\right)}^{2},
\]
which implies that 
\[
A_{N}\sqrt{\lambda}=\frac{1}{N}\sum_{k=1}^{N}\widehat{S}^{k}\sqrt{\lambda}\in\mathscr{H}\left(\lambda\right),
\]
and $\mathscr{H}\left(\lambda\right)$ is closed in $\mathscr{H}\left(X\right)$.
Therefore, $A_{N}\sqrt{\lambda}\longrightarrow\widehat{E}_{1}\sqrt{\lambda}$,
$\widehat{E}_{1}\sqrt{\lambda}=\sqrt{d\mu_{\infty}}$, and $\mu_{\infty}R=\mu_{\infty}$. 
\end{proof}

\section{$\mathscr{L}_{1}\left(R\right)$ as a subspace of $\mathscr{L}\left(R\right)$}

In the present section we study Radon-Nikodym properties of the path-space
measures from Sections \ref{sec:gm} and \ref{sec:el}.
\begin{lem}
Suppose $\lambda_{1}R=\lambda_{1}$, and $\mu\ll\lambda_{1}$; then
$\mu\in\mathscr{L}\left(R\right)$, i.e., we have $\mu R\ll\mu$. 
\end{lem}
\begin{proof}
Let $\frac{d\mu}{d\lambda_{1}}=W$, $W\in L^{1}\left(\lambda_{1}\right)$
and set 
\[
Q\left(x\right)=\begin{cases}
\frac{W\left(\sigma\left(x\right)\right)}{W\left(x\right)} & \mbox{if }W\left(x\right)>0\\
0 & \mbox{if }W\left(x\right)=0
\end{cases};
\]
then 
\begin{align*}
\int f\,d\left(\mu R\right) & =\int R\left(f\right)d\mu=\int_{\left\{ x:W\left(x\right)>0\right\} }R\left(f\right)Wd\lambda_{1}\\
 & =\int R\left(f\left(W\circ\sigma\right)\right)d\lambda_{1}\\
 & =\int f\frac{W\circ\sigma}{W}Wd\lambda_{1}\;(\mbox{since }\lambda_{1}R=\lambda_{1})\\
 & =\int f\,Q\,d\mu.
\end{align*}
Thus, $\mu R\ll\mu$ and $\frac{d\mu R}{d\mu}=Q$. 
\end{proof}
\begin{rem}
There are several interesting questions as to whether there is an
inverse implication. In ``most'' cases of $\mu$ satisfying $\mu R\ll\mu$,
then $\exists\mu_{\infty}\in M_{+}$ s.t. $\mu_{\infty}R=\mu_{\infty}$,
and $\mu_{\infty}\ll\mu$. (See \exaref{2m2}.)
\begin{lem}
Let $\left(X,\mathscr{B},\sigma,R\right)$, $R\mathbbm{1}=\mathbbm{1}$,
be as usual. Suppose $\lambda\in M_{1}$, $\mu\in\mathscr{L}\left(R\right)$,
and $\frac{d\lambda R}{d\lambda}=W$; then 
\begin{align}
1= & \int Wd\lambda=\int\left(W\circ\sigma\right)Wd\lambda=\cdots\label{eq:L1}\\
 & \cdots=\int\left(W\circ\sigma^{n}\right)\cdots\left(W\circ\sigma\right)Wd\lambda=1.\nonumber 
\end{align}
More over, the following GENERAL estimates hold:
\begin{eqnarray}
\int\sqrt{W}d\lambda & \leq & 1\nonumber \\
\int\sqrt{\left(W\circ\sigma\right)W}d\lambda & \leq & 1\nonumber \\
 & \vdots\label{eq:L2}\\
\int\sqrt{\left(W\circ\sigma^{n}\right)\cdots\left(W\circ\sigma\right)W}d\lambda & \leq & 1\nonumber 
\end{eqnarray}
\end{lem}
\end{rem}
\begin{proof}
We note that $\widehat{S}$ is isometric, where 
\begin{align*}
\widehat{S}\bigl(\sqrt{\lambda}\bigr) & =\sqrt{\lambda R}=\sqrt{W}\sqrt{\lambda}\\
\widehat{S}^{2}\bigl(\sqrt{\lambda}\bigr) & =\sqrt{\left(W\circ\sigma\right)W}\sqrt{\lambda}.
\end{align*}
So $1=\bigl\Vert\sqrt{d\lambda}^{2}\bigr\Vert=\bigl\Vert\widehat{S}\sqrt{\lambda}\bigr\Vert^{2}=\bigl\Vert\widehat{S}^{2}\sqrt{\lambda}\bigr\Vert^{2}=\cdots=1$,
i.e., (\ref{eq:L1}) holds.
\end{proof}
\begin{rem}
The conditions in (\ref{eq:L2}) are satisfied in \exaref{2m2}, where
$\lambda=dx=$ Lebesgue measure, 
\[
W\left(x\right)=2\cos^{2}\left(\pi x\right)=1+\cos\left(2\pi x\right),
\]
and $\int W\left(x\right)dx=1$. So
\[
\frac{1}{N+1}\sqrt{\left(W\circ\sigma^{n}\right)\cdots\left(W\circ\sigma\right)W}\xrightarrow[\;N\rightarrow\infty\;]{}0\quad\mbox{in \ensuremath{L^{2}\left(\lambda\right)=L^{2}\left(dx\right)}.}
\]
\end{rem}

\section{\label{sec:me}Multiresolutions from endomorphisms and solenoids}

We now return to a more detailed analysis of the multi-scale resolutions
introduced in  \secref{gm} above. 

General setting: Let $\left(X,\mathscr{B},\sigma,R\right)$, $R\mathbbm{1}=\mathbbm{1}$,
be as usual. 

Multiresolution $\longleftrightarrow$ wavelets (\subsecref{gmwr}),
with levels of resolution given by
\begin{gather}
\mathscr{B}\supseteq\sigma^{-1}\left(\mathscr{B}\right)\supseteq\sigma^{-2}\left(\mathscr{B}\right)\supseteq\cdots\supseteq\mathscr{B}_{\infty},\label{eq:m1}\\
\sigma^{-i}\left(\mathscr{B}\right)=\underset{\text{details in between}}{\underbrace{\left[\sigma^{-i}\left(\mathscr{B}\right)\backslash\sigma^{-\left(i+1\right)}\left(\mathscr{B}\right)\right]}}\cup\sigma^{-\left(i+1\right)}\left(\mathscr{B}\right);\label{eq:m2}
\end{gather}
i.e., $\mathscr{B}$ is space of initial resolution, $\sigma^{-1}\left(\mathscr{B}\right)$
contains less information; see, e.g., \cite{MR1887500,MR1913212,MR2220205,MR3428850,MR3452783,MR3484392,MR3447092}

If $\mu\in\mathscr{L}\left(R\right)\cap\mathscr{K}_{1}$, then we
have the following resolution decomposition:
\begin{equation}
\cdots\subset L_{\mathscr{H}}^{2}\left(\mu R^{2}\right)\subset L_{\mathscr{H}}^{2}\left(\mu R\right)\subset L_{\mathscr{H}}^{2}\left(\mu\right)\label{eq:m3}
\end{equation}

Note that $\widehat{S}\widehat{S}^{*}=\widehat{E}_{1}=$ projection
onto $\widehat{S}\mathscr{H}\left(X\right)$, but if restrict to $\mathscr{H}\left(\mu\right)$,
it is the projection onto $\mathscr{H}\left(\mu R\right)$. 

Recall that, by \corref{O1}, we have $\mathbb{E}^{\left(\mu\right)}\left(f\mid\sigma^{-1}\left(\mathscr{B}\right)\right)=\widehat{S}\widehat{S}^{*}\big|_{\mathscr{H}\left(\mu\right)}\left(f\right)$,
i.e., 
\begin{equation}
\widehat{S}\widehat{S}^{*}\left(f\sqrt{\mu}\right)=\mathbb{E}^{\left(\mu\right)}\left(f\mid\sigma^{-1}\left(\mathscr{B}\right)\right)\sqrt{\mu}\label{eq:m4}
\end{equation}
so that 
\[
\underset{\widehat{E}_{i}}{\underbrace{\widehat{S}^{i}\widehat{S}^{*i}}}\left(f\sqrt{\mu}\right)=\mathbb{E}^{\left(\mu\right)}\left(f\mid\sigma^{-i}\left(\mathscr{B}\right)\right)\sqrt{\mu},\quad i=1,2,3\cdots.
\]

\begin{thm}
Wavelet decomposition for $h\in\mathscr{H}\left(\mu\right)$: 
\begin{align}
h= & k_{0}+\left(k_{1}\circ\sigma\right)\sqrt{W}+\left(k_{2}\circ\sigma^{2}\right)\sqrt{W\left(W\circ\sigma\right)}+\cdots\nonumber \\
 & \cdots+\left(k_{n}\circ\sigma^{n}\right)\sqrt{W\left(W\circ\sigma\right)\left(W\circ\sigma^{2}\right)\cdots\left(W\circ\sigma^{n-1}\right)}+\cdots+h_{\infty}\label{eq:m5}
\end{align}
as a decomposition in $\mathscr{H}\left(\mu\right)\sim L^{2}\left(\mu\right)$. 
\end{thm}
\begin{proof}
We use Wold on $\widehat{S}\big|_{\mathscr{H}\left(\mu\right)}$ as
an isometry. (See \cite{MR1913212,MR2641594,MR1711343,MR616145}.)
Note, if $\mu R\ll\mu$, then $\widehat{S}\mathscr{H}\left(\mu\right)\subset\mathscr{H}\left(\mu\right)$
so that it is isometric in $\mathscr{H}\left(\mu\right)\sim L^{2}\left(\mu\right)$. 

If $\lambda R=1$ then $Sf=f\circ\sigma$ is isometric in $L^{2}\left(\lambda\right)$,
and $S^{*}=R$ relative to the $L^{2}\left(\lambda\right)$-inner
product, so $RS\Big|_{L^{2}\left(\lambda\right)}=I$, $SR=SS^{*}=E_{1}=$
the projection onto $SL^{2}\left(\lambda\right)$, and $E_{1}L^{2}\left(\lambda\right)=SL^{2}\left(\lambda\right)$,
$I-E_{1}=$ the projection onto 
\[
\left(SL^{2}\left(\lambda\right)\right)^{\perp}=\ker S^{*}=\ker R=\left\{ f\in L^{2}\left(\lambda\right)\mid Rf=0\right\} .
\]

The orthogonal expansion for $f\in L^{2}\left(\lambda\right)$ is
as follows:
\[
f=h_{0}+Sh_{1}+S^{2}h_{2}+\cdots+S^{n}h_{n}+\cdots+h_{\infty},
\]
and by Parseval's identity:
\[
\left\Vert f\right\Vert _{L^{2}\left(\lambda\right)}^{2}=\sum_{n=0}^{\infty}\left\Vert h_{n}\right\Vert _{L^{2}\left(\lambda\right)}^{2}+\left\Vert h_{\infty}\right\Vert _{L^{2}\left(\lambda\right)}^{2}.
\]
Note in $L^{2}\left(\lambda\right)$, 
\begin{align*}
SRf & =\mathbb{E}^{\left(\lambda\right)}\left(f\mid\sigma^{-1}\left(\mathscr{B}\right)\right),\\
S^{n}R^{n}f & =\mathbb{E}^{\left(\lambda\right)}\left(f\mid\sigma^{-n}\left(\mathscr{B}\right)\right),\\
E_{\infty}f & =\mathbb{E}^{\left(\lambda\right)}\left(f\mid\mathscr{B}_{\infty}\right),
\end{align*}
where $\mathscr{B}_{\infty}=\cap_{n=1}^{\infty}\sigma^{-n}\left(\mathscr{B}\right)$. 
\end{proof}

\section{\label{sec:AE}Application to Examples \ref{exa:2m1} \& \ref{exa:2m2}}

We now return to a more detailed analysis of the two examples from
 \secref{te} above.
\begin{example}[See Ex \ref{exa:2m1}]
\label{exa:A1}Consider $X=\mathbb{R}/\mathbb{Z}\simeq[0,1)$, $\lambda=dx=$
Lebesgue measure on $\left[0,1\right]$.

Set $\sigma\left(x\right)=2x$ mod 1, $Sf\left(x\right)=f\left(2x\right)$
in $L^{2}\left(\left[0,1\right],\lambda\right)$. Then $S^{*}=R$,
\begin{equation}
\left(Rf\right)\left(x\right)=\frac{1}{2}\Bigl(f\Bigl(\frac{x}{2}\Bigr)+f\Bigl(\frac{x+1}{2}\Bigr)\Bigr),\label{eq:A1}
\end{equation}
and $S^{*}k=0$ $\Longleftrightarrow$ $\exists h$ s.t. $k\left(x\right)=e_{1}\left(x\right)h\left(2x\right)=e_{1}\left(x\right)h\left(\sigma\left(x\right)\right)$,
where $h\in L^{2}\left(\lambda\right)$, and $e_{1}\left(x\right)=e^{i2\pi x}$. 
\end{example}
\begin{prop}
\label{prop:A1}For all $f\in L^{2}\left(\lambda\right)$, there a
unique orthogonal expansion:
\begin{align*}
f\left(x\right)= & e_{1}\left(x\right)h_{0}\left(2x\right)+e_{1}\left(3x\right)h_{1}\left(2^{2}x\right)+\cdots\\
 & \cdots+e_{1}\left(\left(2^{n}-1\right)x\right)h_{n}\left(2^{n}x\right)+\cdots+\mbox{const};
\end{align*}
and 
\[
\left\Vert f\right\Vert _{\lambda}^{2}=\sum_{n=0}^{\infty}\left\Vert h_{n}\right\Vert _{\lambda}^{2}+\left\Vert f_{\infty}\right\Vert _{\lambda}^{2},\quad f_{\infty}=\mbox{const.}
\]
\end{prop}
In the general case we get, for all $h\in\mathscr{H}\left(\mu\right)$:
\[
h=k_{0}\sqrt{\mu}+\left(k_{1}\circ\sigma\right)\sqrt{\mu R}+\left(k_{2}\circ\sigma^{2}\right)\sqrt{\mu R^{2}}+\cdots+h_{\infty},
\]
where $h_{\infty}\in\cap_{i}\widehat{S}^{i}\mathscr{H}\left(\mu\right)$.
See \figref{mra}, and also Sections \ref{subsec:gmmr}, \ref{subsec:mul}.

\begin{figure}[H]
\includegraphics[width=0.7\textwidth]{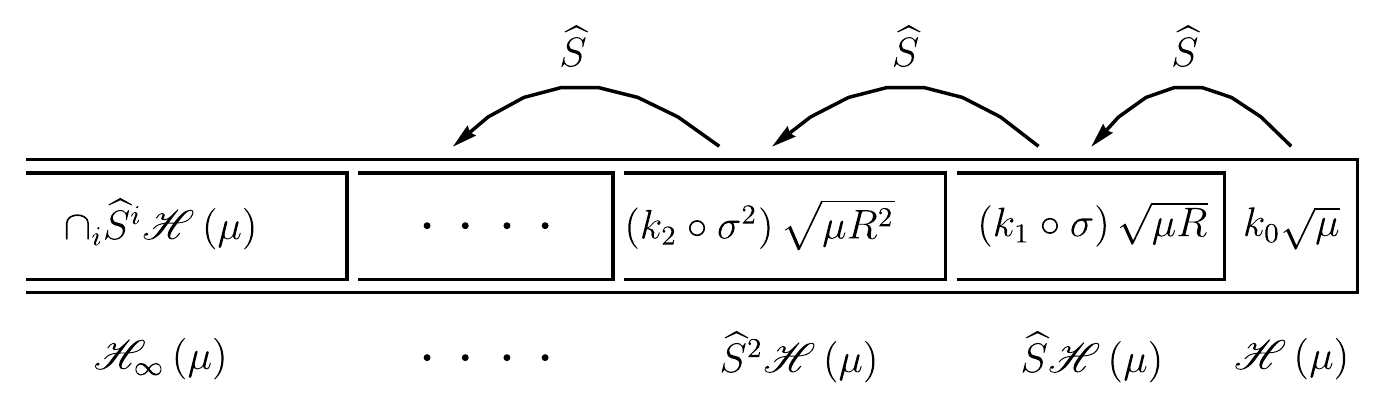}

\caption{\label{fig:mra}Multiresolution expansion.}

\end{figure}

\begin{example}[Ex \ref{exa:A1} continued]
Let $R$ be as in (\ref{eq:A1}). In the real case, we have two solutions
to $Rf=0$: 
\[
f_{c}\left(x\right)=\cos\left(2\pi x\right),\quad f_{s}\left(x\right)=\sin\left(2\pi x\right).
\]
Allowing complex functions we have 
\[
e_{\pm}\left(x\right)=e^{\pm i2\pi x}.
\]
We also check directly that $R^{*}=S$ with 
\[
Sf\left(x\right)=f\left(2x\:\mbox{mod }1\right).
\]
Since $X=\mathbb{R}/\mathbb{Z}$, the functions $f_{c}$, $f_{s}$,
$e_{\pm}$ are $\mathbb{Z}$-periodic and therefore functions on $\mathbb{R}/\mathbb{Z}\simeq[0,1)$. 

Let $\lambda=dx=$ Lebesgue measure on $[0,1)$, i.e., the Haar measure
on $X=\mathbb{R}/\mathbb{Z}$. 
\end{example}
\begin{lem}
We have 
\begin{gather}
\left\langle Rf,g\right\rangle _{\lambda}=\left\langle f,Sg\right\rangle _{\lambda},\;\mbox{i.e.,}\label{eq:am1}\\
\int_{0}^{1}\left(Rf\right)\left(s\right)g\left(x\right)dx=\int_{0}^{1}f\left(x\right)g\left(2x\right)dx,\;\forall f,g\in L^{2}\left(\lambda\right),\label{eq:am2}
\end{gather}
and with $\sigma\left(x\right)=2x\:\mbox{mod }1$. 
\end{lem}
\begin{proof}
Set 
\[
\tau_{0}\left(x\right)=\frac{x}{2},\quad\tau_{1}\left(x\right)=\frac{x+1}{2}
\]
so that $\sigma\left(\tau_{i}\left(x\right)\right)=x$, $\forall x$,
$i=1,2$. Then
\begin{align*}
\mbox{LHS}_{\left(\ref{eq:am2}\right)} & =\int_{0}^{1}\frac{1}{2}\Bigl(f\Bigl(\frac{x}{2}\Bigr)+f\Bigl(\frac{x+1}{2}\Bigr)\Bigr)g\left(x\right)dx\\
 & =\int_{0}^{1}\left[\left(f\left(g\circ\sigma\right)\right)\Bigl(\frac{x}{2}\Bigr)+\left(f\left(g\circ\sigma\right)\right)\Bigl(\frac{x+1}{2}\Bigr)\right]dx\\
 & =\int_{0}^{1}f\left(x\right)g\left(\sigma\left(x\right)\right)dx=\mbox{RHS}_{\left(\ref{eq:am2}\right)}.
\end{align*}
\end{proof}
In \propref{A1}, we have proved that the functions $e_{\pm}\left(x\right)=e^{\pm i2\pi x}$
yield the representation
\begin{equation}
L^{2}\left(\lambda\right)\ni f\left(x\right)=\sum_{n=0}^{\infty}e_{\pm}\left(\left(2^{n}-1\right)x\right)h_{n}^{\left(\pm\right)}\left(2^{n}x\right)+\left(\mbox{const}\right),\label{eq:am3}
\end{equation}
where $h_{n}^{\left(\pm\right)}\in L^{2}\left(\lambda\right)$, so
functions on $\mathbb{R}/\mathbb{Z}$, i.e., $\mathbb{Z}$-periodic
$L^{2}$-functions. This is the multiresolution orthogonal expansion
for $f\in L^{2}\left(\lambda\right)=L^{2}\left(\mathbb{R}/\mathbb{Z},dx\right)$. 

But (\ref{eq:am3}) is the expansion in the complex Hilbert space
$L^{2}\left(\lambda\right)$. In the real case, we get instead, 
\begin{align}
f\left(x\right)= & \sum_{n=0}^{\infty}\cos\left(2\pi\left(2^{n}-1\right)x\right)h_{n}\left(2^{n}x\right)\label{eq:am4}\\
 & +\sum_{n=0}^{\infty}\sin\left(2\pi\left(2^{n}-1\right)x\right)h_{n}\left(2^{n}x\right)+\mbox{const.}\nonumber 
\end{align}

\begin{acknowledgement*}
The co-authors thank the following colleagues for helpful and enlightening
discussions: Professors Daniel Alpay, Sergii Bezuglyi, Ilwoo Cho,
Carla Farsi, Elizabeth Gillaspy, Judith Packer, Wayne Polyzou, Myung-Sin
Song, and members in the Math Physics seminar at The University of
Iowa. The two D. A. and S. B. are recent co-authors of the present
first named author (P. J.)

\bibliographystyle{amsalpha}
\bibliography{ref}
\end{acknowledgement*}

\end{document}